\documentclass{article}
\usepackage{amsthm,amsmath,amssymb}
\RequirePackage[numbers]{natbib}
\RequirePackage[colorlinks,citecolor=blue,urlcolor=blue]{hyperref}

\RequirePackage[OT1]{fontenc}
\RequirePackage[numbers]{natbib}
\RequirePackage[colorlinks,citecolor=blue,urlcolor=blue]{hyperref}
\usepackage{amsthm,amsmath,amssymb}

\usepackage{graphicx, subfigure}
\usepackage{amsfonts, amsmath, amssymb, amsthm, constants}

\usepackage{dsfont}
\usepackage{mathrsfs}
\usepackage{verbatim}

\usepackage{url}

\usepackage{color}

\usepackage[top=1in,bottom=1in,left=1in,right=1in]{geometry}
\usepackage{comment}

\usepackage{cases}
\usepackage{color}
\definecolor{darkred}{RGB}{100,0,0}
\definecolor{darkgreen}{RGB}{0,100,0}
\definecolor{darkblue}{RGB}{0,0,150}

\usepackage{hyperref}
\hypersetup{colorlinks=true, linkcolor=darkred, citecolor=darkgreen, urlcolor=darkblue}
\usepackage{url}

\newtheorem{Theorem}{Theorem}[section]

\newtheorem{lemma}[Theorem]{Lemma}
\newtheorem{Corollary}[Theorem]{Corollary}
\newtheorem{proposition}[Theorem]{Proposition}

\def\beq{\begin{equation}}
\def\eeq{\end{equation}}
\def\beqn{\begin{eqnarray*}}
\def\eeqn{\end{eqnarray*}}
\def\bitem{\begin{itemize}}
\def\eitem{\end{itemize}}
\def\benum{\begin{enumerate}}
\def\eenum{\end{enumerate}}
\def\bmult{\begin{multline*}}
\def\emult{\end{multline*}}
\def\bcenter{\begin{center}}
\def\ecenter{\end{center}}

\def\cA{\mathcal{A}}
\def\cB{\mathcal{B}}
\def\cC{\mathcal{C}}

\def\cK{\mathcal{K}}

\def\cP{\mathcal{P}}

\def\cT{\mathcal{T}}

\def\cV{\mathcal{V}}
\def\cW{\mathcal{W}}

\def\cZ{\mathcal{Z}}

\def\bA{\boldsymbol{A}}
\def\bB{\boldsymbol{B}}

\def\bD{\boldsymbol{D}}
\def\bE{\boldsymbol{E}}

\def\bJ{\boldsymbol{J}}

\def\bQ{\boldsymbol{Q}}

\def\bT{\boldsymbol{T}}
\def\bU{\boldsymbol{U}}
\def\bV{\boldsymbol{V}}

\def\bb{\mathbf{b}}

\def\bTheta{\boldsymbol{\Theta}}

\def\bbE{\mathbb{E}}

\def\bbP{\mathbb{P}}

\def\bbR{\mathbb{R}}

\def\bb\bU{\mathbb{\bU}}

\newcommand{\E}{\operatorname{\mathbb{E}}}
\renewcommand{\P}{\operatorname{\mathbb{P}}}
\newcommand{\Var}{\operatorname{Var}}

\def\\bUnif{\text{\bUnif}}

\newcommand{\1}{\mathds{1}}

\usepackage{ulem}
\begin{document}
\title{Oracle inequalities for network models and sparse graphon estimation}
 \author{{\it Olga Klopp},   CREST and MODAL'X, University Paris Ouest\\
{\it Alexandre B. Tsybakov}, ENSAE, UMR CNRS 9194\\
{\it Nicolas Verzelen, INRA.}}
  
\maketitle

%
%
\begin{abstract}

Inhomogeneous random graph models encompass many network models such as stochastic block models and latent position models.
 We consider the problem of statistical estimation of the matrix of connection probabilities based on the observations of the adjacency matrix of the network. Taking the stochastic block model as an approximation, we  construct estimators of network connection probabilities -- the ordinary block constant least squares estimator, and its restricted version.   We show that they satisfy oracle inequalities with respect to the block constant oracle. As a consequence, we derive optimal rates of estimation of the probability matrix. Our results cover the important setting of sparse networks.   Another consequence consists in establishing upper bounds on the minimax risks for graphon estimation in the $L_2$ norm when the probability matrix is sampled according to a graphon model. These bounds  include an additional term accounting for the ``agnostic" error induced by the variability of the latent unobserved variables of the graphon model.  In this setting, the optimal rates are influenced not only by the bias and variance components as in usual nonparametric problems but also include the third component, which is the agnostic error. The results shed light on the differences between estimation under the empirical loss (the probability matrix estimation) and under the integrated loss (the graphon estimation).
\end{abstract}



\section{Introduction}
 Consider a network defined as an undirected graph with $n$ nodes.   Assume that we observe the values $\bA_{ij}\in \{0,1\}$ where  $\bA_{ij}=1$ is interpreted as the fact that the nodes $i$ and $j$ are connected and  $\bA_{ij}=0$ otherwise. We set $\bA_{ii}=0$ for all $1\le i\le n$ and we assume that $\bA_{ij}$ is a Bernoulli random variable with parameter $(\bTheta_0)_{ij}=\P(\bA_{ij}=1)$ for $1\le j<i\le n$. The random variables $\bA_{ij}$, $1\le j<i\le n$, are assumed independent.  We denote by $\bA$ the adjacency matrix i.e., the $n\times n$ symmetric matrix with entries $\bA_{ij}$ for $1\le j<i\le n$ and zero diagonal entries. Similarly, we denote by $\bTheta_0$ the $n\times n$ symmetric matrix with entries $(\bTheta_0)_{ij}$ for $1\le j<i\le n$ and zero diagonal entries. This is a matrix of probabilities associated to the graph; the nodes $i$ and $j$ are connected with probability $(\bTheta_0)_{ij}$. The model with such observations $\bA'=(\bA_{ij}$, $1\le j<i\le n)$  is a special case of inhomogeneous random graph model that we will call for definiteness the {\it network sequence model}\footnote{In some recent papers, it is also called the inhomogeneous Erd\"os-R\'enyi model, which is somewhat ambiguous since the words ``Erd\"os-R\'enyi model'' designate a homogeneous graph.}, to emphasize a parallel with the gaussian sequence model.

\subsection{Graphons and sparse graphon models}

Networks arise in many areas such as information technology,  social life, genetics. These real-life networks are in permanent movement and often their size is growing. Therefore, it is natural to look for a well-defined "limiting object" independent of the network size $n$ and such that a stochastic network can be viewed as a partial observation of this limiting object.
Such objects called the \textit{graphons} play a central role in the recent theory of graph limits introduced by Lov\'asz and Szegedy \cite{LovaszSzegedy}. For a detailed description of this theory we refer to the monograph by Lov\'asz \cite{LovaszBook}. Graphons are symmetric  measurable functions $W :[0,1]^{2}\rightarrow [0,1]$.
In the sequel, the space of graphons is denoted by $\cW$.
The main message is that every graph limit can be represented by a graphon. This beautiful theory of graph limits was developed for \textit{dense graphs}, that is, graphs with the number of edges comparable to the square of the number of vertices.

 Graphons  give a natural way of generating random graphs~\cite{LovaszBook,MR2463439}. The probability that  two distinct nodes $i$ and $j$ are connected in the graphon model is the random variable 
 \begin{equation}\label{grapnon_mod}
(\bTheta_0)_{ij}= W_0(\xi_i,\xi_j)
\end{equation} 
where $\xi_1,\dots,\xi_n$ are unobserved (latent) i.i.d. random variables uniformly distributed on $[0,1]$. As above, the diagonal entries of $\bTheta_0$ are zero.        
 Conditionally on ${\boldsymbol\xi}=(\xi_1,\dots,\xi_n)$, the observations $\bA_{ij}$ for $1\le j<i\le n$ are assumed to be independent Bernoulli random variables with success probabilities  $(\bTheta_0)_{ij}$. For any positive integer $n$, a graphon function $W_{0}$  defines  a probability distribution on graphs of size $n$. Note that this model is different from the network sequence model since the observations $\bA_{ij}$ are no longer independent. 
 If $W_0$ is a step-function with $k$ steps, we obtain an analog of the stochastic block model with $k$ groups.
 More generally, many exchangeable distributions on networks~\cite{MR2463439}, including random dot product graphs~\cite{MR3113816} and some  geometric random graphs~\cite{MR1986198} can be expressed as graphons.
 

Given an observed adjacency matrix $\bA'$ sampled according to the model \eqref{grapnon_mod}, the graphon function $W_{0}$ is not identifiable. The topology of a network is invariant with respect to any change of labelling of its nodes. Consequently, for any given function $W_{0}(\cdot,\cdot)$ and a measure-preserving bijection $\tau : [0,1]\rightarrow [0,1]$ (with respect to the Lebesgue measure), the functions $W_{0}(x,y)$ and $W_{0}^\tau(x,y):=W_{0} (\tau(x),\tau(y))$ define the same probability distribution on random graphs. The equivalence classes of graphons defining the same probability distribution are characterized in a slightly more involved way based on the following result. 

\begin{proposition}[{\cite[Sect.10]{LovaszBook}}] Two graphons $U$  and $W$ in $\cW$ are called  weakly isomorphic if there exist measure preserving maps $\phi$, $\psi$: $[0,1]\to [0,1]$ such that $U^{\phi}= W^{\psi}$ almost everywhere. Two graphons $U$ and $W$ define the same probability distribution if and only if they are weakly isomorphic. 
\end{proposition}\label{prp:weakiso}

This proposition motivates considering the equivalence classes of graphons that are weakly isomorphic. The corresponding quotient space is denoted by $\widetilde{\cW}$.

It is easy to see that the expected number of edges in model \eqref{grapnon_mod} is a constant times the squared number of vertices, which corresponds to the dense case. Networks observed in the real life are often \textit{sparse} in the sense that the number of edges is of the order $o(n^{2})$ as
the number $n$ of vertices tends to infinity. The situation in the sparse case  is more complex than in the dense case. Even the almost dense case with $n^{2-o(1)}$ edges is rather different from the dense case. The extremely sparse case
corresponds to the \textit{bounded degree} graph where all degrees are smaller than a fixed positive integer. It is an opposite extreme to the dense graph. Networks that occur in the applications are usually between these two extremes described by dense and bounded degree graphs. They often correspond to inhomogeneous networks  with density of edges tending to $0$ but with the maximum degree tending to infinity as $n$ grows.

For a given $\rho_n>0$, one can modify the definition \eqref{grapnon_mod} to get a random graph model with $O(\rho_nn^{2})$ edges. It is usually assumed that $\rho_n\rightarrow 0$ as $n\rightarrow \infty$. 
The adjacency matrix $\bA'$ is  sampled  according to graphon $W_0\in \cW$ with scaling parameter $\rho_n$ if for all $j<i$, 
\begin{equation}\label{sparse_grapnon_mod}
      (\bTheta_0)_{ij}=\rho_n W_0(\xi_i,\xi_j)\ 
\end{equation} 
where $\rho_n>0$ is the scale parameter that can be interpreted as the expected proportion of non-zero edges.
Alternatively, model \eqref{sparse_grapnon_mod} can be considered  as a graphon model \eqref{grapnon_mod} that has been sparsified in the sense that its edges have been independently removed  with probability $1-\rho_n$  and kept with probability $\rho_n$. This sparse graphon model was considered in \cite{bickel2009nonparametric, bickel2011method,WolfeOlhede,xu2014edge}.

%


\bigskip

\subsection{Our results}

This paper has two contributions. First, we study optimal rates of estimation of the probability matrix $\bTheta_0$ under the Frobenius norm from an observation  $\bA'=(\bA_{ij}$, $1\le j<i\le n)$ in the network sequence model. We estimate $\bTheta_0$ by a  block-constant matrix and we focus on deriving oracle inequalities with optimal rates (with possibly non-polynomial complexity of estimation methods). Note that estimating $\bTheta_0$ by a $k\times k$ block constant matrix is equivalent to fitting a stochastic block model with $k$ classes.
Estimation of $\bTheta_0$ has already been considered by \cite{chatterjee2014matrix,xu2014edge, chan2014consistent} but convergence rates obtained there are far from being optimal. More recently,  Gao et al. \cite{gao2014rate} have established the minimax estimation rates for $\bTheta_0$  on classes of block constant matrices and on the smooth graphon classes. Their analysis is restricted to the dense case \eqref{grapnon_mod} corresponding to $\rho_n=1$ when dealing with  model \eqref{sparse_grapnon_mod}. In this paper, we explore the general setting of model \eqref{sparse_grapnon_mod}. In particular, our aim is to understand the behavior of least squares estimators when the probabilities $\left (\bTheta_0\right )_{ij}$ can be arbitrarily small. This will be done via developing oracle inequalities with respect to the block constant oracle.
Two estimators will be considered - the ordinary block constant least squares estimator, and its restricted version where the estimator is chosen in the $\ell_\infty$ cube of a given radius. As corollaries, we provide an extension for the {\it sparse} graphon model of some minimax results in \cite{gao2014rate} and we quantify the impact of the scaling parameter $\rho_n$ on the optimal rates of convergence. 

Second, we consider estimation of the  graphon function $W_{0}$ based on  observation $\bA'$. In view of Proposition \ref{prp:weakiso}, graphons are not identifiable and can only be  estimated up to weak isomorphisms. Hence, we study estimation of $W_{0}$ in the quotient space $\widetilde{\cW}$ of graphons.  In order to contrast this problem with the estimation of $\bTheta_0$, one can invoke an analogy with the random design nonparametric regression. Suppose that we observe $(y_i,\xi_i)$, $i=1,\ldots, n$, that are independently sampled according to the model $y=f(\xi)+\epsilon$ where $f$ is an unknown regression function, $\epsilon$ is a zero mean random variable and $\xi$ is distributed with some density $h$ on $[0,1]$. Given a sample of $(y_i,\xi_i)$, estimation of $f$ with respect to the {\it empirical} loss is equivalent to estimation of the vector $(f(\xi_1),\ldots,f(\xi_n))$ in,  for instance, the Euclidean norm. On the other hand, estimation under the integrated loss consists in constructing an estimator $\hat{f}$ such that the integral $\int (\hat{f}(t)-f(t))^2 h(t)dt$ is small. Following this analogy, estimation of $\bTheta_{0}$ corresponds to an empirical loss problem whereas the graphon estimation corresponds to an integrated loss problem. However, as opposed to nonparametric regression,
in the graphon models~(\ref{grapnon_mod}) and~(\ref{sparse_grapnon_mod}) the design $\xi_1,\dots,\xi_n$ is not observed, which makes it quite challenging to derive the convergence rates in these settings.

 In Section~\ref{sec:graphon}, we obtain $L_2$ norm non-asymptotic estimation rates for graphon estimation on classes of step functions (analogs of stochastic block models) and on classes of smooth graphons in model~\eqref{sparse_grapnon_mod}. This result improves upon previously known bounds by Wolfe and Olhede \cite{WolfeOlhede}. For classes of step function graphons, we also provide a matching minimax lower bound allowing one to characterize the regimes such that graphon optimal estimation rates are slower than probability matrix estimation rates.
In a work parallel to ours, Borgs et al. \cite{borgs2015} have analyzed the  rates of convergence for estimation of step function graphons under the privacy model. Apart from the issues of privacy that we do not consider here, their results in our setting provide a weaker version of the upper bound in Corollary \ref{cor:integrated_risk_sbm}, with a suboptimal rate, which is the square root of the rate of Corollary \ref{cor:integrated_risk_sbm} in the moderately sparse zone.  We also mention the paper by Choi \cite{choi2015} devoted to the convergence of empirical risk functions associated to the graphon model.



\subsection{Notation}

We provide a brief summary of the notation used throughout this paper. 
\begin{itemize}
\item For a matrix $\bB$, we denote by $\bB_{ij}$ (or by $\bB_{i,j}$, or by $(\bB)_{ij}$) its $ (i, j)$th entry.

\item For an integer $m$, set $[m]=\{1,\dots,m\}$.

\item Let  $n,k$ and $n_0$ be integers such that $2\le k\le n$, $n_0\le n$. We denote by $\cZ_{n,k,n_0}$ the set of all mappings $z$ from $[n]$ to $[k]$ such that $\min_{a=1,\ldots, k}|z^{-1}(a)|\geq n_0$ (the minimal size of each ``community" is $n_0$). For brevity, we set $\cZ_{n,k}=\cZ_{n,k,1}$.

\item  We denote by $\mathbb{R}^{k\times k}_{\rm sym}$ the class of all symmetric $k\times k$ matrices with real-valued entries.
\item The inner product between matrices $\bD,\bB\in \mathbb{R}^{n\times n}$ will be denoted by $\langle \bD, \bB\rangle = \sum_{i,j} \bD_{ij}\bB_{ij}$.

\item Denote by $\|B\|_F$ and by $\|B\|_\infty$ the Frobenius norm and the entry-wise supremum norm of matrix $\bB\in \mathbb{R}^{n\times n}$ respectively.

\item We denote by $\lfloor x\rfloor$  the maximal integer less than $x\geq 0$ and by $\lceil x \rceil$  the smallest integer greater than or equal to $x$.
\item We denote by $\E$ the expectation with respect to the distribution of $\bA$ if we consider  the network sequence model and the expectation with respect to the joint distribution of $({\boldsymbol\xi},\bA)$ if we consider the graphon model.  
\item   We denote by $C$ positive constants that can vary from line to line. These are absolute constants unless otherwise mentioned. 
\item We denote by $\lambda$ the Lebesgue measure on the interval $[0,1]$.

\end{itemize}


\section{Probability matrix estimation}\label{sec:proba_matrix}
In this section, we deal with the network sequence model and we obtain optimal rates of estimation in the Frobenius norm of the probability matrix $\bTheta_0$.
Fixing some integer $k>0$,  we estimate the matrix $\bTheta_0$ by a block constant matrix with $k\times k$ blocks.

\subsection{Least squares estimator} 
First, we study the least squares estimator $\widehat{\bTheta}$ of $\bTheta_0$ in the collection of all $k\times k$ block constant matrices with block size larger than some given integer~$n_0$. 
Recall that we denote by $\cZ_{n,k,n_0}$ the set of all possible mappings $z$ from $[n]$ to $[k]$ such that $\min_{a=1,\ldots, k}|z^{-1}(a)|\geq n_0$.
  For any $z\in \cZ_{n,k,n_0}$, $\bQ\in \mathbb{R}^{k\times k}_{\rm sym}$, we define the residual sum of squares 
\[L(\bQ,z)= \sum_{(a,b)\in [k]\times [k]} \ \sum_{(i,j)\in z^{-1}(a)\times z^{-1}(b),\   {j<i}   }(\bA_{ij}-\bQ_{ab})^2
\]
and consider the least squares estimator of $(\bQ,z)$:
\[(\hat{\bQ},\hat{z})\in \arg\min_{\bQ\in \mathbb{R}^{k\times k}_{\rm sym},\ z\in \cZ_{n,k,n_0}}L(\bQ,z)
\]
The block constant least squares estimator of  $(\bTheta_{0})_{ij}$ is defined as $\widehat{\bTheta}_{ij}= \widehat{\bQ}_{\hat{z}(i)\hat{z}(j)}$ for all $i> j$. Note that $\widehat{\bTheta}_{ij}\in[0,1]$. Finally, we denote by $\widehat{\bTheta}$ the symmetric matrix with entries $\widehat{\bTheta}_{ij}$ for all $i> j$ and with zero diagonal entries. 
According to the stochastic  block models terminology, $\hat{\bQ}$ stands for the estimated connection probabilities   whereas $\hat{z}$ is the estimated partition of the of nodes into $k$ communities. 
The only difference between this estimator and the one considered in \cite{gao2014rate} 
is the restriction $|z^{-1}(a)|\geq n_0$. 
This prevents the partition $z$ from being too unbalanced. A common choice of $n_0$ is of the order $n/k$, which is referred to as balanced partition.
 Taking larger $n_0$ when we have additional information allows us to obtain simpler estimators since we reduce the number of configurations.

\subsection{Restricted least squares estimator}

Given $r\in(0,1]$, consider the least squares estimator restricted to the $l_\infty$ ball of radius $r$,
\[(\hat{\bQ}_r,\hat{z}_r)\in \arg\min_{\bQ\in \mathbb{R}^{k\times k}_{\rm sym}:\ \|\bQ\|_{\infty}\leq r,\ z\in \cZ_{n,k}}L(\bQ,z).\]
The estimator $\widehat{\bTheta}^{r}$ of matrix $\bTheta_0$ is defined as the symmetric matrix with entries $\widehat{\bTheta}^{r}_{ij}= (\widehat{\bQ}_r)_{\hat{z}_r(i)\hat{z}_r(j)}$ for all $i> j$, and with zero diagonal entries.
Note that here we consider any partitions, including really unbalanced ones $(n_0=1)$.

\subsection{Oracle inequalities}

Let $\bTheta_{*,n_0}$ be the best Frobenius norm approximation of $\bTheta_0$ in the collection of matrices
$$\cT_{n_0}[k]= \{\bTheta:\ \exists \ z\in \cZ_{n,k,n_0},\ \bQ\in \mathbb{R}^{k\times k}_{\rm sym} \text{ such that } \bTheta_{ij}=\bQ_{z(i)z(j)}, \ i\ne j,   \ \text{ and }  \bTheta_{ii}=0 \ \forall  i\}.$$
In particular, $\cT_{1}[k]$ is the set of all probability matrices corresponding to $k$-class stochastic block models without group size restriction. For brevity, we  write $\cT_{1}[k] = \cT[k]$, and  $\bTheta_{*,1}=\bTheta_*$.
\begin{proposition}\label{prp:l2} Consider the network sequence model.
There exist positive absolute constants $C_1$, $C_2$, $C_3$ such that for $n_0\ge 2$,
\beq\label{eq:risk_l2}
 \E\left[\frac{1}{n^{2}}\|\widehat{\bTheta}- \bTheta_0\|_F^2\right]\leq \frac{C_1}{n^{2}} \|\bTheta_0-\bTheta_{*,n_0}\|_F^2 + C_2\|\bTheta_0\|_{\infty}\left (\frac{\log(k)}{n}+ \frac{k^2}{n^{2}}\right ) +   \frac{C_3\log(n/n_0)}{n_0}\left (\frac{\log(k)}{n}+ \frac{k^2}{n^{2}}\right ).
 \eeq
\end{proposition}

To discuss some consequences of this oracle inequality, we consider the case of balanced partitions.

 \begin{Corollary}\label{cor:1} Consider the network sequence model.
 Let $n_0\ge Cn/k$ for some $C>0$ (balanced partition), $n_0\ge 2$, and  $\|\bTheta_{0}\|_{\infty} \ge \frac{k\log(k)}{n}$. Then,
 there exist positive absolute constants $C_1$, $C_2$, such that  
\beq\label{eq:risk_l2_1}
 \E\left[\frac{1}{n^{2}}\|\widehat{\bTheta}- \bTheta_0\|_F^2\right]\leq \frac{C_1}{n^{2}} \|\bTheta_0-\bTheta_{*,n_0}\|_F^2 + C_2\|\bTheta_0\|_{\infty}\left (\frac{\log(k)}{n}+ \frac{k^2}{n^{2}}\right ).
  \eeq
 \end{Corollary}

In particular, if $\bTheta_0 \in {\mathcal T_{n_0}}[k]$ (i.e., when we have a $k$-class stochastic block model with balanced communities), the rate of convergence is of the order $\|\bTheta_0\|_{\infty}\left (\frac{\log(k)}{n}+ \frac{k^2}{n^{2}}\right )$. In \cite{gao2014rate}, the rate $\left (\frac{\log(k)}{n}+ \frac{k^2}{n^{2}}\right )$ is shown to be minimax over all stochastic block models with $k$ blocks, with no restriction on $\|\bTheta_0\|_{\infty}$ except the obvious one, $\|\bTheta_0\|_{\infty}\le 1$. We prove in the next subsection that the rate $\|\bTheta_0\|_{\infty}\left (\frac{\log(k)}{n}+ \frac{k^2}{n^{2}}\right )$ obtained in Corollary \ref{cor:1} exhibits the optimal dependency on $\|\bTheta_0\|_{\infty}$. Before doing this, we provide an oracle inequality for the restricted least squares estimator $\widehat{\bTheta}^{r}$.

\begin{proposition}\label{prp:l2_Thres} Consider the network sequence  model.
There exist positive absolute constants $C_1$ and $C_2$ such that the following holds. If $\|\bTheta_0\|_{\infty}\le r$, then \beq\label{eq:risk_l2thres}
 \E\left[\frac{1}{n^{2}}\|\widehat{\bTheta}^{r}- \bTheta_0\|_F^2\right]\leq \frac{C_1}{n^{2}} \|\bTheta_0-\bTheta_*\|_F^2 + C_2 r\left (\frac{\log(k)}{n}+ \frac{k^2}{n^{2}}\right ) \ .
 \eeq
 \end{proposition}

As opposed to Corollary \ref{cor:1}, the risk bound \eqref{eq:risk_l2thres} is applicable  for unbalanced partitions and for arbitrarily small values of $\|\bTheta_{0}\|_{\infty}$. However, this restricted least squares estimator requires the knowledge of an upper bound of $\|\bTheta_0\|_{\infty}$. Whereas the mean value  of $(\bTheta_0)_{ij}$ is easily inferred from the data, the maximal value $\|\bTheta_0\|_{\infty}$ is difficult to estimate. If the matrix $\bTheta_0$ satisfies the sparse graphon model \eqref{sparse_grapnon_mod}, one can set $r=r_n= u_n \overline{A}$ where $\overline{A}=2\sum_{j<i}\bA_{i,j}/(n(n-1))$ is the edge density of the graph and $u_n$ is any sequence that tends to infinity (for example, $u_n=\log\log n$). For $n$ large enough, $\overline{A}/\rho_n$ is close to $\int_{[0,1]^2} W_{0}(x,y){\rm d}x{\rm d}y$ in the sparse graphon model \eqref{sparse_grapnon_mod} with probability close  to one, and therefore $r_n$ is  greater than $\|\bTheta_0\|_{\infty}$. The price to pay for this data-driven choice of $r_n$ is that the rate in the risk bound \eqref{eq:risk_l2thres} is multiplied by $u_n$.

%
%
\subsection{Stochastic block models}

 Given an integer $k$ and any $\rho_{n}\in (0,1]$, consider the set of all probability matrices corresponding to $k$-class stochastic block model with connection probability uniformly smaller than $\rho_n$:
\beq\label{definition_T_k_rho}
\cT[k,\rho_{n}]= \left\{\bTheta_0\in \cT[k]  :\ \|\bTheta_0\|_{\infty}\leq \rho_{n}\right\}.
\eeq
  Gao et al.~\cite{gao2014rate} have proved that the minimax estimation rate over $\cT[k,1]$ is of order the $\tfrac{k^2}{n^2}+\tfrac{\log(k)}{n}$. The following proposition extends their lower bound to arbitrarily small $\rho_{n}>0$.
 
\begin{proposition}\label{lower_sparse_sbm} Consider the network sequence  model. 
For all $k\leq n$ and all $0<\rho_{n}\le 1$,  
  \beq
 \inf_{\widehat{\bT}}\sup_{\bTheta_0\in \cT[k,\rho_n]}\E_{\bTheta_0}\left[\frac{1}{n^{2}}\|\widehat{\bT}-\bTheta_0\|_F^2\right]
 \geq  
 C\min \Big(\rho_{n}\Big(\frac{\log(k)}{n}+ \frac{k^2}{n^{2}}\Big), \rho_{n}^2\Big)
 \eeq
 where $\E_{\bTheta_0}$ denotes the expectation with respect to the distribution of $\bA$ when the underlying probability matrix is $\bTheta_0$
 and $\inf_{\widehat{\bT}}$ is the infimum over all estimators.
\end{proposition}

If $\rho_{n}\geq \tfrac{\log(k)}{n}+ \tfrac{k^2}{n^{2}}$, the minimax rate of estimation is of the order $\rho_{n}\left (\frac{\log(k)}{n}+ \frac{k^2}{n^{2}}\right )$. This rate is achieved by the restricted least squares estimator with $r\asymp \rho_{n}$ and by the least squares estimator if the partition is balanced and $\rho_{n} \geq k\log(k)/n$. For really sparse graphs ($\rho_{n}$ smaller than $\tfrac{\log(k)}{n}+ \tfrac{k^2}{n^{2}}$), the estimation problem becomes rather trivial since both the null  estimator $\widehat{\bT}=0$ and the constant least squares estimator $\widehat{\bTheta}$ with all entries   $\widehat{\bTheta}_{ij}= \overline{A}$ achieve the optimal rate $\rho_{n}^2$.

\subsection{Smooth graphons}

We now use Propositions \ref{prp:l2} and \ref{prp:l2_Thres} to obtain rates of convergence for the probability matrix estimation when $\bTheta_0$ satisfies the sparse graphon model \eqref{sparse_grapnon_mod} and the graphon $W_{0}\in \cW$ is smooth.  For any $\alpha>0$, $L>0$, define the class of $\alpha$-H\"{o}lder continuous functions $\Sigma(\alpha, L)$ as the set of all functions $W:[0,1]^2\to [0,1]$ such that 
$$
\vert W(x',y')-\cP_{\lfloor \alpha\rfloor}((x,y), (x'-x,y'-y))   \vert \le L (\vert x'-x\vert^{\alpha-\lfloor \alpha\rfloor}+\vert y'-y\vert^{\alpha-\lfloor \alpha\rfloor})
$$
for all $(x',y'), (x,y) \in [0,1]^2$, where $\lfloor \alpha\rfloor$ is the maximal integer less than $\alpha$, and the function $(x',y')\mapsto \cP_{\lfloor \alpha\rfloor}((x,y), (x'-x,y'-y))$ is the Taylor polynomial of degree $\lfloor \alpha\rfloor$ at point $(x,y)$. It follows from the standard embedding theorems that, for any $W\in \Sigma(\alpha, L)$ and any $(x',y'), (x,y) \in [0,1]^2$,
\begin{equation}\label{lip}
\vert W(x',y') - W(x,y)\vert \le M (\vert x'-x\vert^{\alpha\wedge 1}+\vert y'-y\vert^{\alpha\wedge 1})
\end{equation}
where $M$ is a constant depending only on $\alpha$ and $L$. In the following, we will use only this last property of $W\in \Sigma(\alpha, L)$.

\smallskip

The following two propositions give bounds on the bias terms $\|\bTheta_0-\bTheta_{*,n_0}\|_F^2$ and $\|\bTheta_0-\bTheta_{*}\|_F^2$.

\begin{proposition}\label{approx_sbm_balanced} Consider the graphon model \eqref{sparse_grapnon_mod} with $W_{0}\in \Sigma(\alpha, L)$ where $\alpha, L>0$.
Let $n_0\geq 2$ and $k=\lfloor n/n_0\rfloor$. Then, 
  \beq\label{eq:risk_SBM_approx_balanced}
 \E\left (\frac{1}{n^{2}}\|\bTheta_0-\bTheta_{*,n_0}\|_F^2\right )\leq CM^{2}\rho^{2}_{n}\left (\dfrac{1}{k^{2}}\right )^{\alpha\wedge 1}.
  \eeq
  \end{proposition} 
We will also need the following proposition proved in \cite{gao2014rate}, Lemma 2.1.

\begin{proposition}
\label{prp:gao}
Consider the graphon model \eqref{sparse_grapnon_mod}
with  $W_{0} \in \Sigma(\alpha, L)$ where $\alpha, L>0$.
   Then,
   almost surely,
  \beq\label{eq:risk_SBM_approx}
 \dfrac{1}{n^{2}}\|\bTheta_0- \bTheta_*\|_F^2\leq CM^{2}\rho^{2}_{n}\left (\dfrac{1}{k^{2}}\right )^{\alpha\wedge 1}.
  \eeq
  \end{proposition} 
  
  \begin{Corollary}\label{prp:empirical_smooth_raphon} Consider the graphon model \eqref{sparse_grapnon_mod} with   $W_{0}\in \Sigma(\alpha, L)$ where $\alpha, L>0$ and $4/n<\rho_{n}\le 1$. Set $k=\Big\lceil \left (\rho_n^{1/2}n\right )^{\frac{1}{1+\alpha\wedge 1}} \Big\rceil$ if $\rho_n\geq n^{(\alpha\wedge 1)-1}$  and $k=\Big\lceil \left (\rho_n n\right )^{\frac{1}{2(\alpha\wedge 1)}} \Big\rceil$ if $\rho_n< n^{(\alpha\wedge 1)-1}$
  	\begin{itemize}
  		\item[(i)] Assume that $\rho_n\geq n^{-\frac{2(\alpha\wedge 1)}{1+2(\alpha\wedge 1)}}\left (\log (n\rho_n) \right )^{\frac{2(1+\alpha\wedge 1)}{1+2(\alpha\wedge 1) }}$	  and there exists $n_0\ge 2$ such that $k=\lfloor n/n_0\rfloor$. Then
  		there exists a positive constant $C$ depending only  on $L$ and $\alpha$ such that the least squares estimator $\widehat{\bTheta}$ constructed with this choice of $n_0$ satisfies
  		\beq\label{eq:risk_empirical_smooth_1}
  		\E\left[\dfrac{1}{n^{2}}\|\widehat{\bTheta} -\bTheta_0\|_F^2 \right] \leq C\left \{ \rho_{n}^{\frac{2+\alpha\wedge 1}{1+\alpha\wedge 1}}n^{-\frac{2(\alpha\wedge 1)}{1+\alpha\wedge 1}} + \dfrac{\rho_n\log  (\rho_n n)}{n}\right\} .
  		\eeq
  		\item[(ii)] Assume that $r\geq \rho_{n}\geq Cn^{-1}$. Then
  		there exists a positive constant $C$ depending only on $L$ and $\alpha$ such that the restricted least squares estimator $\widehat{\bTheta}^r$  satisfies
  		\beq\label{eq:risk_empirical_smooth}
  		\E\left[\dfrac{1}{n^{2}}\|\widehat{\bTheta}^r -\bTheta_0\|_F^2 \right] \leq C\left \{ r^{\frac{2+\alpha\wedge 1}{1+\alpha\wedge 1}}n^{-\frac{2(\alpha\wedge 1)}{1+\alpha\wedge 1}} + \dfrac{r\log (\rho_n n)}{n}\right\} .
  		\eeq
  	\end{itemize}
  	
  \end{Corollary}
  Corollary \ref{prp:empirical_smooth_raphon} extends the results obtained in Gao et al.~\cite{gao2014rate} to arbitrary $\rho_n\in (0,1]$.  To simplify the discussion assume that $\alpha\leq 1$. There are two
  ingredients in the rates of convergence in Corollary~\ref{prp:empirical_smooth_raphon}, the nonparametric rate $\rho_n^{\frac{2+\alpha}{1+\alpha}}n^{-\frac{2\alpha}{1+\alpha}}$ and the clustering rate $\tfrac{\rho_n\log (\rho_n n)}{n}$.  The smoothness index $\alpha$  has an impact on the rate only if $\alpha\in (0,1)$ and only if the network is not too sparse, that is if $\rho_n\geq Cn^{\alpha-1}(\log \rho_n n))^{1+\alpha}$.
  
  In~\cite{gao2014rate}, Gao et al. prove a lower bound showing that the rate $n^{-\frac{2(\alpha\wedge 1)}{1+\alpha\wedge 1}}+  \tfrac{\log(n)}{n}$ is optimal if $\rho_n=1$. Following the same lines as in Proposition \ref{lower_sparse_sbm}, one can readily extend their result to prove that the rate in \eqref{eq:risk_empirical_smooth} is minimax optimal for $\rho_n\geq Cn^{-1}$.

\section{Graphon estimation problem}\label{sec:graphon}
 In this section, our  purpose is to estimate the graphon function $W_{0}(\cdot,\cdot)$ in the sparse graphon model \eqref{sparse_grapnon_mod}.
 \subsection{From probability matrix estimation to graphon estimation}
We start by associating a graphon to any $n\times n$ probability matrix $\bTheta$. This provides a way of deriving an estimator of $f_{0}(\cdot,\cdot)= \rho_n W_{0}(\cdot,\cdot)$ from an estimator of $\bTheta_0$.


 Given a $n\times n$ matrix $\bTheta$ with entries in $[0,1]$, define the {\it empirical graphon} $\widetilde{f}_{\bTheta}$ as the following piecewise constant function:
 \beq\label{eq:empirical_graphon}
 \widetilde{f}_{\bTheta}(x,y)= \bTheta_{\lceil nx\rceil, \lceil ny \rceil}
 \eeq
 for all  $x$ and $y$ in $(0,1]$.  
%
 The empirical graphons associated to the least squares estimator and to the  restricted least squares estimator with threshold $r$ will be denoted by $\widehat f$  and $\widehat f_r$ respectively:
 $$
 \widehat{f} =  \widetilde f_{\widehat\bTheta}\,, \qquad \widehat f_r = \widetilde f_{{\widehat \bTheta}^r}.
 $$
 They will be used as graphon estimators.
 For any estimator $\check f$ of $f_0=\rho_n W_0$, define the squared error
\begin{equation}\label{definition_delta_2}
\delta^2(\check{f},f_0):= \underset{\tau\in \mathcal M}{\inf}\int \int_{(0,1)^{2}}\vert f_0(\tau(x),\tau(y))-\check f(x,y)\vert^{2}\mathrm{d}x\mathrm{d}y
\end{equation}
 where $\mathcal M$ is the set of all measure-preserving bijections $\tau : [0,1]\rightarrow [0,1]$. It has been proved in~\citep[Ch.8,13]{LovaszBook} that    $\delta(\cdot,\cdot)$ defines a metric on the quotient space $\widetilde{\cW}$ of  graphons.  
 
 \smallskip
 
 The following lemma is a simple consequence of the triangle inequality. 
\begin{lemma}\label{lem:upper_risk_graphon} Consider the graphon model \eqref{sparse_grapnon_mod}.
For any $W_0\in \cW$, $\rho_n>0$, and any estimator $\widehat{\bT}$ of $\bTheta_0$ such that $\widehat{\bT}$ is an $n\times n$ matrix with entries in $[0,1]$, we have
\beq\label{eq:integrated_risk_decomposition}
\E\left[\delta^2(\widetilde{f}_{\widehat{\bT}}, f_0)\right]\leq 2\E\left[\frac{1}{n^2}\|\widehat{\bT}-\bTheta_0\|_F^2\right] +2\E\left[\delta^2\left(\widetilde{f}_{\bTheta_0} , f_0\right)\right].
\eeq
\end{lemma}
 The bound on the integrated risk in \eqref{eq:integrated_risk_decomposition}  is a sum of two terms. The first term containing $\|\widehat{\bT}-\bTheta_0\|_F^2$  has been considered in Section \ref{sec:proba_matrix} for $\widehat{\bT}=\widehat{\bTheta}$ and $\widehat{\bT}=\widehat{\bTheta}^r$. It is the {\it estimation error} term. The second term containing $\delta^2(\widetilde{f}_{\bTheta_0} , f_0)$  measures the distance between the true graphon $f_0$ and its discretized version sampled at the unobserved random design points $\xi_1,\ldots , \xi_n$. We call it the {\it agnostic error}. The behavior of $\delta^2(\widetilde{f}_{\bTheta_0} , f_0)$ depends on the topology of the considered graphons as shown below for two examples, namely, the step function graphons and the smooth graphons.

\subsection{Step function graphons}
 
Define $\cW[k]$ the collection of $k$-step graphons, that is the subset of graphons $W\in \cW$ such that for some $\bQ\in \bbR^{k\times k}_{\text{sym}}$ and some $\phi:[0,1]\to [k]$, 
\beq\label{eq:def_step_function}
W(x,y)= \bQ_{\phi(x),\phi(y)}\quad \text{  for all }x,y\in [0,1]\ .
\eeq
A step function $W\in \cW[k]$ is called balanced if $\lambda\left (\phi^{-1}(1)\right )= \ldots =  \lambda\left (\phi^{-1}(k)\right )=1/k$ where $\lambda$ is the Lebesgue measure on $[0,1]$. The agnostic error associated to step function graphons is evaluated as follows.

\begin{proposition}\label{prp:bias_sbm} Consider the graphon model \eqref{sparse_grapnon_mod}.
For all integers $k\le n$,  $W_0\in \cW[k]$ and $\rho_n>0$ we have
\[\E\left[\delta^2\left(\widetilde{f}_{\bTheta_0} , f_0\right)\right]\leq C \rho_n^2  \sqrt{\frac{k}{n}}\ .
\]
\end{proposition}

Combining this result with Lemma~\ref{lem:upper_risk_graphon} and Propositions \ref{prp:l2} and \ref{prp:l2_Thres} we obtain the following risk bounds for the least squares and restricted least squares graphon estimators. 

\begin{Corollary}\label{cor:integrated_risk_sbm} Consider the graphon model \eqref{sparse_grapnon_mod} with $W_0\in \cW[k]$.
There exist absolute constants $C_1$ and $C_2$ such that the following holds. 
\begin{itemize}
\item[(i)] If $\rho_n \ge \frac{k\log(k)}{n}$, $n_0= \lfloor n/2k \rfloor$, 
and the step function $W_0$ is balanced, then the least squares  graphon estimator $\widehat{f}$  constructed with this choice of $n_0$ satisfies 
\[
\E\left[\delta^2\left(\widehat{f} , f_0\right)\right]\leq C_1\left[  \rho_n \left(\frac{k^2}{n^2}+  \frac{\log(k)}{n}\right)+ \rho_n^2  \sqrt{\frac{k}{n}}\right]. 
\]
\item[(ii)]
If $\rho_n\leq r$, the restricted least squares graphon estimator $\widehat{f}_r$ satisfies 
\beq\label{eq:integred_risk_lbm_restricted}
 \E\left[\delta^2\left(\widehat{f}_{r} , f_0\right)\right]\leq C_2  \left[r\left(\frac{k^2}{n^2}+  \frac{\log(k)}{n}\right)+  \rho_n^2  \sqrt{\frac{k}{n}}\right]. 
\eeq
\end{itemize}
\end{Corollary}

As an immediate consequence of this corollary, we get the following upper bound on the minimax risk:
\beq\label{eq:minimax_risk_graphon}
\inf_{\widehat{f}}\sup_{W_0\in \cW[k]}\E_{W_0}\left[\delta^2\left(\widehat{f}, f_0\right)\right]
\leq C_3  
 \left\{\left[\rho_n \left(\frac{k^2}{n^2}  +  \frac{\log(k)}{n}\right) + \rho_n^2  \sqrt{\frac{k}{n}} \right] \wedge \rho_n^2\right\}
\eeq
where $C_3$ is an absolute constant. Here, $\E_{W_0}$ denotes the expectation with respect to the distribution of observations $\bA'=(\bA_{ij}, 1\le j<i\le n)$ when the underlying sparse graphon is $\rho_n W_0$
 and $\inf_{\widehat{f}}$ is the infimum over all estimators. The bound \eqref{eq:minimax_risk_graphon} follows from \eqref{eq:integred_risk_lbm_restricted} with $r= \rho_n$ and the fact that the risk of the null estimator $\tilde f\equiv 0$ is smaller than $\rho_n^2$. 
 
The next proposition shows that the upper bound \eqref{eq:minimax_risk_graphon} is optimal in a minimax sense (up to a logarithmic factor in $k$ in one of the regimes).

\begin{proposition}\label{prp:lower_minimax_integrated_risk}
There exists a universal constant $C>0$ such that for any sequence $\rho_n>0$ and any positive integer $2\leq k\leq n$, 
\beq\label{eq:lower_minimax_integrated_risk}
 \inf_{\widehat{f}}\sup_{W_0\in \cW[k]}\E_{W_0}\left[\delta^2\left(\widehat{f}, f_0\right)\right]\geq C  
 \left\{\left[\rho_n \left(\frac{k^2}{n^2}  +  \frac{1}{n}\right) + \rho_n^2  \sqrt{\frac{k-1}{n}} \right] \wedge \rho_n^2\right\} ,
 \eeq
 where $\E_{W_0}$ denotes the expectation with respect to the distribution of observations $\bA'=(\bA_{ij}, 1\le j<i\le n)$ when the underlying sparse graphon is $\rho_n W_0$
 and $\inf_{\widehat{f}}$ is the infimum over all estimators. 
\end{proposition}

The proof is given in Section \ref{sec:proof_minimax_integrated_risk}. Showing the  rate $\rho_n^2  \sqrt{(k-1)/n} $ relies on two arguments. First, we construct a family of graphons that are well separated with respect to the $\delta(\cdot,\cdot)$ distance. The difficulty comes from the fact that this is not a classical $L_2$ distance but a minimum of the $L_2$ distance over all measure-preserving transformations. Second, we show that the behavior of the Kullback-Leibler divergences between the graphons in this family is driven by the randomness of the latent variables $\xi_1,\dots,\xi_n$ while the nature of the connection probabilities can be neglected. This argument is similar to the ``data processing'' inequality. 

\smallskip

Proposition \ref{prp:lower_minimax_integrated_risk} does not cover the case $k=1$ (Erd\H{o}s-R\'enyi models). In fact, in this case the constant graphon estimator $\widehat{f}= \widetilde{f}_{\overline{A}}$ achieves the optimal rate, which is equal to $\rho_n/n^2$.  
The suboptimality of the bound of Proposition \ref{prp:bias_sbm} for $k=1$ is due to the fact that the diagonal blocks of $\widehat{f}$ and $\widehat{f}^r$ are all equal to zero. By carefully defining the diagonal blocks of these estimators, we could have achieved the optimal bound $\rho_n^2\sqrt{(k-1)/n}$ but this would need a modification of the notation and of the proofs.

\smallskip

The bounds \eqref{eq:minimax_risk_graphon} and \eqref{eq:lower_minimax_integrated_risk} imply that there are three regimes depending on the sparsity parameter~$\rho_n$:
\begin{itemize}
 \item[(i)] {\it Weakly sparse graphs:} $\rho_n\geq \tfrac{\log(k)}{\sqrt{kn}}\vee (\tfrac{k}{n})^{3/2}$. The minimax risk is of the order $\rho_n^2\sqrt{k/n}$, and thus it is driven by the agnostic error arising from the lack of knowledge of the design.
 
 \item[(ii)] {\it Moderately sparse graphs:} $ \tfrac{\log(k)}{n}\vee \left (\tfrac{k}{n}\right )^{2} \leq \rho_n \leq \tfrac{\log(k)}{\sqrt{kn}}\vee \left (\tfrac{k}{n}\right )^{3/2}$. The risk bound \eqref{eq:minimax_risk_graphon} is driven by the probability matrix estimation error. The upper bound \eqref{eq:minimax_risk_graphon} is of the order $\rho_n \left(\frac{k^2}{n^2}+  \frac{\log(k)}{n}\right)$, which is the optimal rate of probability matrix estimation, cf. Proposition \ref{lower_sparse_sbm}. Due to \eqref{eq:lower_minimax_integrated_risk}, it is optimal up to $\log(k)$ factor with respect to the $\delta(\cdot,\cdot)$ distance. 
 
\item[(iii)] {\it Highly sparse graphs:} $\rho_n\leq  \tfrac{\log(k)}{n}\vee \left (\tfrac{k}{n}\right )^{2}$. The minimax risk is of the order $\rho_n^2$, and it is attained by the null estimator. 
\end{itemize}

\smallskip 

In a work parallel to ours, Borgs et al. \cite{borgs2015} provide an upper bound for the risk of step function graphon estimators in the context of privacy. If the partitions are balanced, Borgs et al. \cite{borgs2015} obtain the bound on the agnostic error as in Proposition \ref{prp:bias_sbm}. When there is no privacy issues, comparing the upper bound of \cite{borgs2015} with that of Corollary \ref{cor:integrated_risk_sbm}, we see that it has a suboptimal rate, which is the square root of the rate of Corollary \ref{cor:integrated_risk_sbm} in the moderately sparse zone. Note also that the setting in \cite{borgs2015} is restricted to balanced partitions while we consider more general partitions.

\subsection{Smooth graphons}

We now derive bounds on the mean squared error of smooth graphon estimation. The analysis will be based on the results of Section \ref{sec:proba_matrix} and on the following bound for the agnostic error associated to smooth graphons.

\begin{proposition}\label{prp:tographon} Consider the graphon model \eqref{sparse_grapnon_mod} with $W_0\in \Sigma(\alpha, L)$ where $\alpha, L>0$~and~$\rho_n>0$.  Then
  \beq\label{eq:risk_L2_SQE}
 \E \left[ \delta^2(\widetilde{f}_{\bTheta_0}, f_0)\right] \leq C\frac{\rho_n^{2}}{n^{\alpha\wedge 1}}\, 
  \eeq
  where the constant $C$ depends only on $L$ and $\alpha$. 
\end{proposition} 

Combining Proposition \ref{prp:tographon} with Lemma~\ref{lem:upper_risk_graphon} and with Propositions \ref{prp:l2} and \ref{prp:l2_Thres} we obtain the following risk bounds for the least squares and restricted least squares graphon estimators.

\begin{Corollary}\label{cor:2}
	Consider the graphon model \eqref{sparse_grapnon_mod} with   $W_0\in \Sigma(\alpha, L)$ where $\alpha, L>0$ and $4/n<\rho_{n}\le 1$. Fix $k=\Big\lceil \left (\rho_n^{1/2}n\right )^{\frac{1}{1+\alpha\wedge 1}} \Big\rceil$  if $\rho_n\geq n^{(\alpha\wedge 1)-1}$  and $k=\Big\lceil \left (\rho_n n\right )^{\frac{1}{2(\alpha\wedge 1)}} \Big\rceil$ if $\rho_n< n^{(\alpha\wedge 1)-1}$ 
	\begin{itemize}
		\item[(i)] Assume that $\rho_n\geq n^{-\frac{2(\alpha\wedge 1)}{1+2(\alpha\wedge 1)}}\left (\log (n\rho_n) \right )^{\frac{2(1+\alpha\wedge 1)}{1+2(\alpha\wedge 1) }}$ and there exists $n_0\ge 2$ such that $k=\lfloor n/n_0\rfloor$. Then,
		there exists a positive constant $C_1$ depending only on $L$ and $\alpha$  such that the least squares graphon estimator $\widehat{f}$ constructed with this choice of $n_0$ satisfies
		\beq\label{eq:cor:2}
		\E\left[\delta^2\left(\widehat{f},f_0\right)\right]  \leq C_1\left \{ \rho_{n}^{\frac{2+\alpha\wedge 1}{1+\alpha\wedge 1}}n^{-\frac{2(\alpha\wedge 1)}{1+\alpha\wedge 1}} + \dfrac{\rho_n\log (\rho_nn)}{n}+  \frac{\rho_n^2}{n^{\alpha\wedge 1}}\right\}  .
		\eeq
		\item[(ii)] Assume that $r\geq \rho_{n}\geq Cn^{-1}$. Then,
		there exists a positive constant $C_2$ depending only on $L$ and $\alpha$ such that the restricted least squares graphon estimator $\widehat{f}_{r}$ satisfies 
		\beq\label{eq:cor:3}
		\E\left[\delta^2\left(\widehat{f}_r,f_0\right)\right] \leq C_2\left \{ r^{\frac{2+\alpha\wedge 1}{1+\alpha\wedge 1}}n^{-\frac{2(\alpha\wedge 1)}{1+\alpha\wedge 1}} + \frac{r\log (\rho_nn)}{n} + \frac{\rho_n^2 }{n^{\alpha\wedge 1}}\right\}  .
		\eeq
	\end{itemize}
\end{Corollary}

For the purpose of the discussion, assume that $r\asymp \rho_n$. 
If $\rho_n\leq  n^{\alpha-1}\log(\rho_n n)$, the  rate of convergence is of the order $\rho_n\log( \rho_n n)/n$, the same as that of the probability matrix estimation risk $\E\big[\|\widehat{\bTheta}^r-\bTheta_0\|_F^2/n^2\big]$, cf. Corollary \ref{prp:empirical_smooth_raphon}. Observe that the above condition is always satisfied when $\alpha\geq  1$. If $\rho_n\geq n^{\alpha-1}\log(\rho_n n)$, the rate of convergence in \eqref{eq:cor:3}  for $\alpha\le 1$ is of the order $\rho_n^2/n^{\alpha}$ due to the agnostic error. This is slower  than the optimal nonparametric rate
$\rho_{n}^{\frac{2+\alpha}{1+\alpha}}n^{-\frac{2\alpha}{1+\alpha}}$  for probability matrix estimation. We conjecture that this loss is unavoidable when considering 
graphon estimation with the $\delta^2(\cdot,\cdot)$ error measure. We also note that the rates in \eqref{eq:cor:3} are  faster than those obtained by Wolfe and Olhede \cite{WolfeOlhede} for the maximum likelihood estimator. In some cases,  the improvement in the rate is  up to $n^{-\alpha/2}$.

\section{Proofs}

In this section, we will repeatedly use Bernstein's inequality that we state here for reader's convenience: Let $X_1,\dots,X_N$ be independent zero-mean random variables. Suppose that $|X_i|\leq M$ almost surely, for all $i$. Then, for any $t>0$,
\begin{equation}\label{bernstein}
\bbP\left \{\sum_{i=1}^{N} X_i\geq \sqrt{2t\sum_{i=1}^{N}\bbE(X^{2}_i)}+\dfrac{2M}{3}t\right \}\leq e^{-t}.
\end{equation}

{In the proofs of the upper bounds, to shorten the notation, we assume that instead of the $n(n-1)/2$ observations $\bA'$
we have the symmetrized observation matrix $\bA$, and thus
\[L(\bQ,z)=  \frac{1}{2} \sum_{(a,b)\in [k]\times [k]} \ \sum_{(i,j)\in z^{-1}(a)\times z^{-1}(b),\   j\ne i   }(\bA_{ij}-\bQ_{ab})^2.
\]
This does not change the estimators nor the results.
The changes are only in the values of constants that do not explicitly appear in the risk bounds.
}

\subsection{Proof of Proposition \ref{prp:l2}}

Since $\widehat{\bTheta}$ is a least squares estimator, we have 
\beq\label{eq:least_squares}
\|\widehat{\bTheta}- \bTheta_0\|_F^2\leq \|\bTheta_0-\bTheta_*\|_F^2+ 2\langle \widehat{\bTheta}-\bTheta_0, \bE\rangle +2\langle \bTheta_0-\bTheta_*, \bE\rangle\ 
\eeq
where $\bE= \bA-\bTheta_0$ is the noise matrix. 
The last summand on the right hand side of \eqref{eq:least_squares} has mean zero. So, to prove the proposition it suffices to bound the expectation of $\langle \widehat{\bTheta}-\bTheta_0, \bE\rangle$. 
For any $z\in \cZ_{n,k,n_0}$, denote by $\widetilde{\bTheta}_z$ the best Frobenius norm approximation of $\bTheta_{0}$ in the collection of matrices
\[
\cT_z:=\{\bTheta: \ \exists \bQ\in \mathbb{R}^{k\times k}_{\rm sym} \text{ such  that } \ \bTheta_{ij}=\bQ_{z(i)z(j)},  \ i\ne j,   \ \text{ and }  \bTheta_{ii}=0 \ \forall  i\}.
\]
We have the following decomposition:
\beqn
\langle \widehat{\bTheta}-\bTheta_0, \bE\rangle= \langle \widetilde{\bTheta}_{\hat{z}}-\bTheta_0, \bE\rangle + \langle \widehat{\bTheta}-\widetilde{\bTheta}_{\hat{z}}, \bE\rangle= (I) + (II)\ .
\eeqn 
In this decomposition, (I)
is the error due to misclustering and (II) is the error due to the Bernoulli noise. We bound each of these errors separately.

\noindent 
{\em Control of (I)}.  We apply Bernstein's inequality \eqref{bernstein} together with the union bound over all $z\in \cZ_{n,k,n_0}$ and we use that  $\langle \widetilde{\bTheta}_{{z}}-\bTheta_0, \bE\rangle=2\sum_{1\leq j < i\leq n} \left (\widetilde{\bTheta}_{{z}}-\bTheta_0\right )_{ij}\bE_{ij}$.  Since the variance of $\bE_{ij}$ satisfies $\Var (\bE_{ij})\le \|\bTheta_0\|_{\infty}$  while $\bE_{ij}\in [-1,1]$, and the cardinality of  $ \cZ_{n,k,n_0}$ satisfies $|\cZ_{n,k,n_0}|\leq k^n$, we obtain  
\begin{eqnarray*}
\P\left[\langle \widetilde{\bTheta}_{\hat{z}}-\bTheta_0, \bE\rangle\geq 2\|\widetilde{\bTheta}_{\hat{z}}-\bTheta_0\|_F\sqrt{\|\bTheta_0\|_{\infty}(n\log(k)+ t)}+ \frac{4}{3} \|\widetilde{\bTheta}_{\hat{z}}-\bTheta_0\|_{\infty}(n\log(k) + t) \right]\leq e^{-t} \  ,
\end{eqnarray*}
for all $t>0$.  Since the entries of $\widetilde{\bTheta}_{\hat{z}}$ are equal to averaged entries of $\bTheta_0$ over blocks, we obtain that $\|\widetilde{\bTheta}_{\hat{z}} - \bTheta_0\|_\infty\leq \|\bTheta_0\|_{\infty}$. Using this observation, decoupling the term $ {2}\|\widetilde{\bTheta}_{\hat{z}}-\bTheta_0\|_F\sqrt{\|\bTheta_0\|_{\infty}(n\log(k)+ t)}$
 via the elementary inequality $2uv\le u^2 + v^2$ and then integrating with respect to $t$ we obtain 
\beq\label{eq:upper_non_parametric1}
\E\left[\langle \widetilde{\bTheta}_{\hat{z}}-\bTheta_0, \bE\rangle - \frac{1}{8}\|\widetilde{\bTheta}_{\hat{z}}-\bTheta_0\|_F^2\right] \leq C \|\bTheta_0\|_{\infty} n\log(k)\ .
\eeq

\noindent 
{\em Control of (II)}. The control of the error due to the Bernoulli noise is more involved. We first consider the intersection of $\cT_z$ with the unit ball in the Frobenius norm 
and construct a $1/4-$net on this set. 
 Then, using the union bound and Bernstein's inequality, we can write a convenient bound on $\langle \bTheta, \bE\rangle$ for any $\bTheta$ from this net. Finally we control (II) using a bound on $\langle \bTheta, \bE\rangle$ on the net and a bound on the supremum norm $\|\widetilde{\bA}_{\hat{z}}-\widetilde{\bTheta}_{\hat{z}}\|_{\infty}$. We control the supremum norm $\|\widetilde{\bA}_{\hat{z}}-\widetilde{\bTheta}_{\hat{z}}\|_{\infty}$ using the definition of $\hat z$ and Bernstein's inequality. 
 
 For any $z\in \cZ_{n,k,n_0}$, define the set
%
$\cT_{z,1}=\{\bTheta\in \cT_z: \ \|\bTheta \|_F\le 1\}$ and denote by ${\widetilde \bA}_z$ the best Frobenius norm approximation of $\bA$ in $\cT_z$. Then, $\widetilde\bE_z = {\widetilde \bA}_z -\widetilde{\bTheta}_z$ is the projection of $\bE$ onto $\cT_z$.
Notice that 
$\tfrac{\widetilde{\bA}_z-\widetilde{\bTheta}_z}{\|\widetilde{\bA}_z-\widetilde{\bTheta}_z\|_F}=\tfrac{\widetilde{\bE}_z}{\|\widetilde{\bE}_z\|_F}$ maximizes $\langle \bTheta', \bE \rangle$ over all $\bTheta'\in \cT_{z,1}$. 
Denote by $\cC_z$ the minimal $1/4$-net on $\cT_{z,1}$ in the Frobenius norm. To each $\bV \in \cC_z$, we associate 
\[\widetilde{\bV} \in\arg\min_{\bTheta' \in \cT_{z,1}\ \cap \  \cB_{\|\cdot\|_F}(\bV ,1/4)}\|\bTheta'\|_{\infty}\  ,\]  which is a matrix minimizing the entry-wise supremum norm over the Frobenius ball $\cB_{\|\cdot\|_F}(\bV ,1/4)$ of radius 1/4 centered at ${\bV}$. Finally, define $\widetilde{\cC}_z:=\{\widetilde{\bV}:\ \bV\in \cC_z\}$.

A standard bound for covering numbers implies that $\log|\widetilde{\cC_z}|\leq Ck^2$, where $|S|$ denotes the cardinality of set $S$. By Bernstein's inequality combined with the union bound, we find that with probability greater than $1-e^{-t}$,
\[\langle \bTheta, \bE\rangle\leq {2}\sqrt{\|\bTheta_0\|_\infty(n\log(k)+ k^2 + t)}+ {\frac{4}{3}}\|\bTheta\|_{\infty}(n\log(k)+k^2+t)\]
simultaneously for all $\bTheta\in \tilde{\cC}_{z}$ with any $z\in \cZ_{n,k,n_0}$. Here, we have used that $\|\bTheta\|_F\le 1$ for all $\bTheta\in \tilde{\cC}_{z}$.

Assume w.l.o.g. that $\widetilde{\bA}_{\hat{z}}-\widetilde{\bTheta}_{\hat{z}} \ne 0$. By the definition of $\tilde{\cC}_{z}$, there exists $\bTheta\in \widetilde{\cC}_{\hat{z}}$, such that  
$$
\|\bTheta - \tfrac{\widetilde{\bA}_{\hat{z}}-\widetilde{\bTheta}_{\hat{z}}}{\|\widetilde{\bA}_{\hat{z}}-\widetilde{\bTheta}_{\hat{z}}\|_F}\|_F\leq \frac1{2}  \quad \text{and} \quad \|\bTheta\|_{\infty}\leq \frac{\|\widetilde{\bA}_{\hat{z}}-\widetilde{\bTheta}_{\hat{z}}\|_{\infty}}{\|\widetilde{\bA}_{\hat{z}}-\widetilde{\bTheta}_{\hat{z}}\|_F}.
$$
Note that for this $\bTheta$, the matrix $2\left(\bTheta - \tfrac{\widetilde{\bA}_{\hat{z}}-\widetilde{\bTheta}_{\hat{z}}}{\|\widetilde{\bA}_{\hat{z}}-\widetilde{\bTheta}_{\hat{z}}\|_F}\right)$ belongs to $\cT_{z,1}$. Thus, 
\beqn
\langle  \tfrac{\widetilde{\bA}_{\hat{z}}-\widetilde{\bTheta}_{\hat{z}}}{\|\widetilde{\bA}_{\hat{z}}-\widetilde{\bTheta}_{\hat{z}}\|_F},\bE\rangle \leq  
\langle \bTheta,\bE\rangle + \frac{1}{2}\ \max_{\bTheta'\in \cT_{\hat{z},1}}\langle \bTheta',\bE\rangle
=
\langle \bTheta,\bE\rangle + \frac{1}{2}\left \langle \tfrac{\widetilde{\bA}_{\hat{z}}-\widetilde{\bTheta}_{\hat{z}}}{\|\widetilde{\bA}_{\hat{z}}-\widetilde{\bTheta}_{\hat{z}}\|_F},\bE \right \rangle
\eeqn
since $\tfrac{\widetilde{\bA}_{\hat{z}}-\widetilde{\bTheta}_{\hat{z}}}{\|\widetilde{\bA}_{\hat{z}}-\widetilde{\bTheta}_{\hat{z}}\|_F}$  maximizes $\langle \bTheta', \bE \rangle$ 
over all $\bTheta'\in \cT_{\hat{z},1}$. Using the last two displays we obtain that 
\beq\label{variance_term1}
\langle  \widetilde{\bA}_{\hat{z}}-\widetilde{\bTheta}_{\hat{z}},\bE\rangle \leq {4}\|\widetilde{\bA}_{\hat{z}}-\widetilde{\bTheta}_{\hat{z}}\|_F\sqrt{\|\bTheta_0\|_{\infty}(n\log(k)+ k^2 + t)}+ {\frac{8}{3}}\|\widetilde{\bA}_{\hat{z}}-\widetilde{\bTheta}_{\hat{z}}\|_{\infty}(n\log(k)+k^2+t)  
\eeq
with probability greater than $1-e^{-t}$.

\bigskip

Next, we are looking for a bound on $\|\widetilde{\bA}_{\hat{z}}-\widetilde{\bTheta}_{\hat{z}}\|_{\infty}$. Since ${\widetilde \bA}_{\hat z} -\widetilde{\bTheta}_{\hat z}=\widetilde\bE_{\hat z}$  we have
\[
\left[\widetilde{\bA}_{\hat z}-\widetilde{\bTheta}_{\hat z}\right]_{ij}= \frac{\sum_{l'\neq l:\ \hat  z(l')=\hat z(i), \hat z(l)=\hat z(j)}\bE_{l'l}}{|(l',l):\  l'\neq l, \hat z(l')=\hat z(i), \hat z(l)=\hat z(j) |} 
\] 
for all $i\neq j$. Consequently, we have
$\|\widetilde{\bA}_{\hat{z}}-\widetilde{\bTheta}_{\hat{z}}\|_{\infty} \leq \sup_{m=n_0,\ldots, n}\sup_{s=n_0,\ldots, n} X_{ms}$ where
\[
X_{ms}:=\sup_{\cV_1: |\cV_1|=m}\ \sup_{\cV_2: |\cV_2|=s}\frac{\left|\sum_{(i,j)\in \cV_1\times \cV_2:\ i\neq j }\bE_{ij}\right |}{ms- |\cV_1\cap \cV_2|}.
\]
Since $n_0\geq 2$, we have $ms- |\cV_1\cap \cV_2|\ge ms - m\wedge s \geq ms/2$ for all $m,s\ge n_0$. Furthermore $|\{\cV_1: |\cV_1|=m\}| \le {n \choose m} \le (en/m)^m$. Therefore,
Bernstein's inequality combined with the union bound over $\cV_1, \cV_2$ leads to
\[
\P\left[X_{ms}
\leq C\left(\sqrt{\|\bTheta_0\|_{\infty}\frac{m\log(en/m)+ s\log(en/s)+t}{ms}} + \frac{m\log(en/m)+ s\log(en/s)+t}{ms}\right)\right]\geq 1-2e^{-t}
\]
for any $t>0$.
From a union bound over all integers $m,s\in [n_0,n ]$, we conclude that 
\beq\label{variance_term2}
\P\left[\|\widetilde{\bA}_{\hat{z}}-\widetilde{\bTheta}_{\hat{z}}\|_{\infty}\leq C\left(\|\bTheta_0\|_{\infty}+ \frac{\log(n/n_0)+t}{n_0} \right) \right]\geq 1-2e^{-t}.
\eeq
Decoupling \eqref{variance_term1} by use of the inequality $2uv\le u^2+ v^2$ and combining the result with \eqref{variance_term2} we obtain 
\beqn
\langle  \widetilde{\bA}_{\hat{z}}-\widetilde{\bTheta}_{\hat{z}},\bE\rangle&\leq& \frac{\|\widetilde{\bA}_{\hat{z}}-\widetilde{\bTheta}_{\hat{z}}\|^2_F}{16}+ C\Big(\|\bTheta_0\|_{\infty}(n\log(k)+ k^2+ t) \\&&  + \frac{\log(n/n_0)(n\log(k)+ k^2) + t(n\log(k)+ k^2+ t) }{n_0}\Big)
\eeqn
with probability greater than $1-3e^{-t}$. Integrating with respect to $t$ leads to
\beq\label{eq:upper_non_parametric2}
\E\left[\langle  \widetilde{\bA}_{\hat{z}}-\widetilde{\bTheta}_{\hat{z}},\bE\rangle-  \frac{\|\widetilde{\bA}_{\hat{z}}-\widetilde{\bTheta}_{\hat{z}}\|^2_F}{16}\right]\leq C\left(\|\bTheta_0\|_{\infty}(n\log(k)+ k^2) +   \frac{\log(n/n_0)(n\log(k)+ k^2)}{n_0} \right) .
\eeq
Now, note that $\widehat{\bTheta} = \widetilde{\bA}_{\hat{z}}$, and hence $\Vert \widetilde{\bTheta}_{\hat{z}}-\bTheta_0\Vert_F\le \Vert \widehat{\bTheta}-\bTheta_0\Vert_F$ by definition of  $\widetilde{\bTheta}_{\hat{z}}$. Thus, $\|\widetilde{\bA}_{\hat{z}}-\widetilde{\bTheta}_{\hat{z}}\|_F = \Vert \widehat{\bTheta} - \widetilde{\bTheta}_{\hat{z}}\Vert_F\le 2\Vert \widehat{\bTheta}-\bTheta_0\Vert_F$.
These remarks, together with \eqref{eq:upper_non_parametric1} and \eqref{eq:upper_non_parametric2}, imply 
\beq\nonumber
\E\left[\langle  \widehat{\bTheta}-\bTheta_{0},\bE\rangle\right]
\le 
\frac{3}{8}\E \|\widehat{\bTheta}-\bTheta_0\|^2_F
+ 
C\left(\|\bTheta_0\|_{\infty}(n\log(k)+ k^2) +   \frac{\log(n/n_0)(n\log(k)+ k^2)}{n_0} \right) .
\eeq
The 
result of the proposition now follows from the last inequality and  \eqref{eq:least_squares}.
\subsection{Proof of Proposition \ref{prp:l2_Thres}}
We first follow the lines of the proof of Proposition~\ref{prp:l2}. As there, we have
\beqn
\|\widehat{\bTheta}^{r}- \bTheta_0\|_F^2\leq \|\bTheta_0-\bTheta_*\|_F^2+ 2\langle \widetilde{\bTheta}_{\hat{z}_r}-\bTheta_0, \bE \rangle+ 2\langle \widehat{\bTheta}^{r}-\widetilde{\bTheta}_{\hat{z}_r}, \bE \rangle+2\langle \bTheta_0-\bTheta_*, \bE\rangle .
\eeqn
Here, $\E \langle \bTheta_0-\bTheta_*, \bE\rangle =0$, and analogously to \eqref{eq:upper_non_parametric1},
\beq\label{eq:upper_non_parametric1_n}
\E\left[\langle \widetilde{\bTheta}_{\hat{z}_r}-\bTheta_0, \bE\rangle - \frac{1}{8}\|\widetilde{\bTheta}_{\hat{z}_r}-\bTheta_0\|_F^2\right] \leq C \|\bTheta_0\|_{\infty} n\log(k).
\eeq
It remains to obtain a bound on the expectation of $\langle \widehat{\bTheta}^{r}-\widetilde{\bTheta}_{\hat{z}_r}, \bE \rangle$, which corresponds to the term (II) in the proof of Proposition~\ref{prp:l2} (the error due to the Bernoulli noise). To do this, we consider the subset $\cA$ of matrices in $\cT_{\hat{z}_r}$ such that their supremum norm is bounded by $2r$ and Frobenius norm is bounded by $\|\widehat{\bTheta}^{r}-\widetilde{\bTheta}_{\hat{z}_r}\|_F$. As $\widehat{\bTheta}^{r}-\widetilde{\bTheta}_{\hat{z}_r}$ belongs to this set it is enough to obtain a bound on the supremum of $\langle\bTheta, \bE \rangle$ for $\bTheta\in \cA$. To control this supremum, we  construct a finite subset $\cC_{z}^*$ that approximates well the set $\cA$ both in the Frobenius norm and in the supremum norm.

  Consider the set
$$
\cA = \{ \bTheta\in \cT_{\hat{z}_r}: \  \|\bTheta\|_{\infty}\leq 2r,\ \|\bTheta\|_F\leq \|\widehat{\bTheta}^{r}-\widetilde{\bTheta}_{\hat{z}_r}\|_F\}.
$$
Then,
\beq\label{eq:definition_azr}
\langle \widehat{\bTheta}^{r}-\widetilde{\bTheta}_{\hat{z}_r}, \bE \rangle\leq \max_{\bTheta\in \cA}\langle \bTheta , \bE\rangle := \langle \hat{\bT},\bE\rangle _F 
\eeq
where $\hat{\bT}$ is a matrix in $\cA$ that achieves the maximum. If $\| \hat{\bT}\|_F < 2\,r$ we have a trivial bound $\langle \hat{\bT},\bE\rangle _F \le 2rn$ since all components of $\bE$ belong to $[-1,1]$. Thus, it suffices to consider the case $\| \hat{\bT}\|_F \ge 2r$.  In order to bound $\langle \hat{\bT},\bE\rangle _F$ in this case, we  construct a finite subset of $\cT_{\hat{z}_r}$ that approximates well  $\hat{\bT}$ both in the Frobenius norm and in the supremum norm. 

For each $z\in \cZ_{n,k}$, let $\cC_z$ be a minimal $1/4$-net of $\cT_{z,1}$ in the Frobenius norm. Set $\epsilon_0= r$ and $\epsilon_q= 2^q\epsilon_0$ for any integer $q=1,\dots, q_{\max},$ where $q_{\max}$ is the smallest integer such that $2rn\le \epsilon_{q_{\max}}$. Cearly, $q_{\max} \le C \log(n)$. 
For any $\bV\in \cC_z$, any $q=0,\dots, q_{\max}$, and any matrix $\bU \in \{-1,0,1\}^{k\times k}$, define a matrix $\bV^{q,\bU,z}\in \mathbb{R}^{n\times n}$ with elements $\bV^{q,\bU,z}_{ij}$ such that $\bV^{q,\bU,z}_{ii}=0$ for all $i\in [n]$, and for all $i\neq j$, 
\beq\label{eq:concentration_bV0}
\bV^{q,\bU,z}_{ij}=\mathrm{sign}(\bV_{ij})\left (|\epsilon_q \bV_{ij}|\wedge (2r) \right)(1-|\bU_{z(i)z(j)}|)+ r\bU_{z(i)z(j)}.
\eeq
Finally, denote by $\cC_{z}^*$ the set  of all such matrices: $$\cC_{z}^*:=\{\bV^{q,\bU,z}: \  \bV\in \cC_z, \ q=0,\dots, q_{\max},\ \bU \in \{-1,0,1\}^{k\times k}\}.$$
For any $z\in \cZ_{n,k}$ we have $\log|\cC_{z}^*|\leq C (k^2 + \log\log(n))$ while $\log | \cZ_{n,k}|\le n\log (k)$. Also, the variance of $\bE_{ij}$ satisfies $\Var (\bE_{ij})\le \|\bTheta_0\|_{\infty}\le r$, and  $\bE_{ij}\in [-1,1]$. Thus, from Bernstein's inequality combined with the union bound we obtain that, with probability greater than $1-e^{-t}$,
$$
\langle \bV,\bE\rangle\leq 2\|\bV\|_F \sqrt{r(k^2 + n\log(k) +t )}+  { \frac{4}{3}}r(k^2 + n\log(k) +t )
$$
simultaneously for all matrices $\bV\in \cC_{z}^*$ and all $z\in \cZ_{n,k}$. Here, we have used that $ \|\bV\|_{\infty}\le 2r$ for all $\bV\in \cC_{z}^*$, cf. \eqref{eq:concentration_bV0}. It follows that, with probability greater than $1-e^{-t}$, 
\beq\label{eq:concentration_bV}
\langle \bV,\bE\rangle - \frac{1}{50}\|\bV\|^2_F \le   Cr(k^2 + n\log(k) +t ) 
\eeq
simultaneously for all matrices $\bV\in \cC_{z}^*$ and all $z\in \cZ_{n,k}$. 

\smallskip

We now use the following lemma proved in Subsection \ref{subsec:lemmas}.
\begin{lemma}\label{lem:covering}
If $\| \hat{\bT}\|_F \ge 2r$ there exists a matrix $\hat{\bV}$ in $\cC_{\hat{z}_r}^*$, such that 
\begin{itemize}
 \item $\|\hat{\bT}-\hat{\bV}\|_F \leq \|\hat{\bT}\|_F/4,$
 \item $\|\hat{\bT}-\hat{\bV}\|_{\infty}\leq r.$
\end{itemize}
\end{lemma}

Thus, on the event $\| \hat{\bT}\|_F \ge r$, as a consequence of Lemma \ref{lem:covering} we obtain that $ 2(\hat{\bT}-\hat{\bV})\in \cA$. This and the definition of $\hat \bT$ imply $\langle \hat{\bT}-\hat{\bV},\bE\rangle\leq \langle \hat{\bT},\bE\rangle _F/2$, so that $\langle \hat{\bT},\bE\rangle _F\leq 2 \langle \hat{\bV},\bE\rangle$, and thus 
$$\langle  \widehat{\bTheta}^{r}-\widetilde{\bTheta}_{\hat{z}_r},\bE\rangle/2\leq \langle \hat{\bV},\bE\rangle.$$
 Furthermore, by Lemma \ref{lem:covering}
 $$\|\hat{\bV}\|_F \leq 5\|\hat{\bT}\|_F/4\le 5\|\widehat{\bTheta}^{r}-\widetilde{\bTheta}_{\hat{z}_r}\|/4.
 $$
These remarks (recall that they hold on the event $\| \hat{\bT}\|_F \ge r$), and \eqref{eq:concentration_bV} yield 
\[\P\left[\Bigg(\langle  \widehat{\bTheta}^{r}-\widetilde{\bTheta}_{\hat{z}_r},\bE\rangle - \frac{\|\widehat{\bTheta}^{r}-\widetilde{\bTheta}_{\hat{z}_r}\|^2_F}{16}\Bigg) \1_{\{\| \hat{\bT}\|_F \ge r\}} \ge Cr(k^2 + n\log(k) +t )\right]\leq e^{-t}\]
for all $t>0$. Integrating with respect to $t$ and using that $\langle  \widehat{\bTheta}^{r}-\widetilde{\bTheta}_{\hat{z}_r},\bE\rangle\le rn$ for $\| \hat{\bT}\|_F < r$ we obtain
\[\E_{\bTheta_0}\left[\langle  \widehat{\bTheta}^{r}-\widetilde{\bTheta}_{\hat{z}_r},\bE\rangle-  \frac{\|\widehat{\bTheta}^{r}-\widetilde{\bTheta}_{\hat{z}_r}\|^2_F}{16}\right]\leq Cr (n\log(k)+ k^2).
\]
In view of this inequality and \eqref{eq:upper_non_parametric1_n} the proof is now finished by the same argument as in Proposition~\ref{prp:l2}.

\subsection{Proof of Proposition \ref{lower_sparse_sbm}}

The proof follows the lines of Theorem 2.2 in \cite{gao2014rate}.  So, for brevity, we only outline the differences from that proof.  The remaining details can be found in \cite{gao2014rate}. First, note that to prove the theorem it suffices to obtain separately the lower bounds of the order $\rho_n (k/n)^2 \wedge \rho_n^2$ and of the order $ \rho_n\log(k)/n \wedge \rho_n^2$. Next, note that the Kullback-Leibler divergence $\cK(\rho_n p, \rho_n q)$ between two Bernoulli distributions with parameters $\rho_n p$ and $\rho_n q$ such that $1/4< p, q<3/4$ satisfies 
\begin{eqnarray}\label{kullb}
\cK(\rho_n p, \rho_n q) &= &\rho_n q \log\left(\frac{q}{p}\right)+ (1-\rho_n q)\log\left(\frac{1-\rho_n q}{1-\rho_n p}\right)\\
& \leq&   \rho_n \frac{q(q-p)}{p}+ \rho_n\frac{1-\rho_n q}{1-\rho_n p}(p-q) = \rho_n \frac{(q-p)^2}{p(1-\rho_n p)} \leq \frac{16}{3}\rho_n (q-p)^2 .
\nonumber
\end{eqnarray}
The main difference from the proof in \cite{gao2014rate} is that now the matrices $\bTheta_0$ defining the probability distributions in the Fano lemma depend on $\rho_n$. Namely, to prove the $\rho_n (k/n)^2 \wedge \rho_n^2$ bound we consider matrices of connection probabilities with elements $\frac{\rho_n}{2} + c_1\sqrt{\rho_n}\left(\frac{k}{n}\wedge\sqrt{\rho_n}\right)\omega_{ab}$ with suitably chosen $\omega_{ab}\in \{0,1\}$ and $c_1>0$ small enough (for $\rho_n=1$ this coincides with the definition of elements $\bQ^\omega_{ab}$ in \cite{gao2014rate}). Then the squared Frobenius distance between matrices $\bTheta_0$ in the corresponding set is of the order $n^2(\rho_n (k/n)^2 \wedge \rho_n^2)$, which leads to the desired rate, whereas in view of \eqref{kullb} the Kullback-Leibler divergences between the probability measures in this set are bounded by $Cn^2\rho_n \left(\frac{k^2}{n^2\rho_n} \wedge 1\right)\le Ck^2$. Thus, the lower bound of the order  $\rho_n (k/n)^2 \wedge \rho_n^2$ follows. 

The argument used to prove the lower bound of the order $ \rho_n\log(k)/n \wedge \rho_n^2$ is quite analogous. To this end, we modify the corresponding construction of  \cite{gao2014rate} only in that we take the connection probabilities of the form $\frac{\rho_n}{2} + c_2\sqrt{\rho_n}\left(\sqrt{\frac{\log(k)}{n}}\wedge\sqrt{\rho_n}\right)\omega_{a}$ with suitably chosen $\omega_{a}\in \{0,1\}$ and $c_2>0$ small enough (for $\rho_n=1$ this coincides with the definition of probabilities $B_a$ in the proof of \cite{gao2014rate}).


\subsection{ Proof of Proposition \ref{approx_sbm_balanced}}

We prove that there exists a random matrix $\overline \bTheta$  measurable with respect to $\xi_1,\dots,\xi_n$ and with values in $\cT_{n_0}[k]$ satisfying
  \beq\label{eq:risk_SBM_approx_balanced2}
 \E\left (\frac{1}{n^{2}}\|\bTheta_0-\overline \bTheta\|_F^2\right )\leq CM^{2}\rho^{2}_{n}\left (\dfrac{1}{k^{2}}\right )^{\alpha\wedge 1}.
  \eeq
Obviously, this implies \eqref{eq:risk_SBM_approx_balanced}. To obtain such $\overline \bTheta$ we construct a balanced partition $z^{*}$ where the first $k-1$ classes contain exactly $n_0$ elements and the last class contains $n-n_0(k-1)$ elements. Then we construct $\overline \bTheta$ using block averages on the blocks given by $z^{*}$.

Let  $n=n_0k+r$ where $r$ is a remainder term between $0$ and $n_0-1$. Define $z^{*}\;:\;[n]\rightarrow[k]$ by 
\[\left (z^{*}\right )^{-1}(a)=\left \{i\in[n]\,:\,\xi_i=\xi_{(j)}\;\text{for some}\;j\in[(a-1)n_0+1,an_0]\right \}\]
for each $a\in\{1,\dots,k-1\}$ and
\[\left (z^{*}\right )^{-1}(k)=\left \{i\in[n]\,:\,\xi_i=\xi_{(j)}\;\text{for some}\;j\in[(k-1)n_0+1,n]\right \}\]
where $\xi_{(j)}$ denotes the $j$th order statistic.
Note that with this partition the first $k-1$ classes contain $n_0$ elements and the last class contains $n_0+r$ elements. We define 
\begin{equation*}
n^{*}_{ab}=\left \{
\begin{array}{llll}
n_0^{2}, & \mbox{if} \quad    a\not=b, a\not=k,b\not=k,\\
(n_0+r)n_0, & \mbox{if} \quad    a=k \,\text{or}\, b=k\,\text{and}\, a\not=b,\\
(n_0-1)n_0, & \mbox{if} \quad    a=b \,\text{and}\, a\not=k,\\
(n_0+r)(n_0+r-1), & \mbox{if} \quad    a=b=k.
\end{array} 
\right.
\end{equation*} 
Using the partition $z^{*}$, we define the block average
\[\bQ^{*}_{ab}=\frac1{n^{*}_{ab}} \ \sum_{i\in (z^{*})^{-1}(a),j\in (z^{*})^{-1}(b),i\not =j}\rho_{n}W(\xi_{i},\xi_{j}).\]
Finally, the approximation $\overline\bTheta$ of $\bTheta_{0}$ 
is defined as a symmetric matrix with entries $\overline\bTheta_{ij}=\bQ^{*}_{z^{*}(i)z^{*}(j)}$ for all $i\ne j$ and $\overline\bTheta_{ii}=0$ for all $i$.
  We have  
\begin{equation*}
\begin{split}
 \E\left (\frac{1}{n^{2}}\|\bTheta_0-\overline\bTheta\|_F^2\right )&= \frac{1}{n^{2}} \sum_{a\in[k],b\in[k]}\E\sum_{i\in (z^{*})^{-1}(a),j\in (z^{*})^{-1}(b),i\not =j}\left (\bTheta_{ij}-\bQ^{*}_{ab}\right )^{2}\\& \hskip -2.5 cm   =
 \frac{1}{n^{2}} \sum_{a\in[k],b\in[k]}\E\sum_{i\in (z^{*})^{-1}(a),j\in (z^{*})^{-1}(b),i\not =j}\left (\rho_{n}W(\xi_{i},\xi_{j})-\dfrac{\sum_{u\in (z^{*})^{-1}(a),v\in (z^{*})^{-1}(b),u\not =v}\rho_{n}W(\xi_{u},\xi_{v})}{n^{*}_{ab}}\right )^{2}.
\end{split}
\end{equation*}
Define $J_{a}= [(a-1)n_0+1,an_0]$ if $a<k$ and $J_k= [(k-1)n_0+1,n]$. By definition of $z^{*}$ we have 
\begin{equation}\label{prop_2_1}
\begin{split}
 \E\left (\frac{1}{n^{2}}\|\bTheta_0-\overline \bTheta\|_F^2\right )&= 
 \frac{\rho_{n}^{2}}{n^{2}} \sum_{a\in[k],b\in[k]}\E\sum_{i\in J_a,j\in J_b,i\not =j}\left (W(\xi_{(i)},\xi_{(j)})-\dfrac{\sum_{u\in J_a,v\in J_b,u\not =v}W(\xi_{(u)},\xi_{(v)})}{n^{*}_{ab}}\right )^{2}\\&\leq 
 \frac{\rho_{n}^{2}}{n^{2}} \sum_{a\in[k],b\in[k]}\sum_{i\in J_a,j\in J_b,i\not =j}\left (\dfrac{1}{n^{*}_{ab}}\sum_{u\in J_a,v\in J_b,u\not =v}\E\left (W(\xi_{(i)},\xi_{(j)})-W(\xi_{(u)},\xi_{(v)})\right )^{2}\right ).
\end{split}
\end{equation}
Using \eqref{lip} and Jensen's inequality we obtain
\begin{equation*}
\begin{split}
\E\left (W(\xi_{(i)},\xi_{(j)})-W(\xi_{(u)},\xi_{(v)})\right )^{2}&\leq M^2\E\left[\left (\vert\xi_{(i)}-\xi_{(u)}\vert^{\alpha'}+\vert\xi_{(j)}-\xi_{(v)}\vert^{\alpha'}\right )^{2}\right]
\\&
\leq 2M^2\left ((\E\vert\xi_{(i)}-\xi_{(u)}\vert^{2})^{\alpha'}
+(\E\vert\xi_{(j)}-\xi_{(v)}\vert^{2})^{\alpha'} \right)
\end{split}
\end{equation*}
where we set for brevity $\alpha'=\alpha\wedge 1$.
Note that by definition of $z^{*}$ we have $\vert i-u\vert < 2n_0$ and $\vert j-v\vert < 2n_0$. Therefore, application of Lemma \ref{lemma_diff_ordered_stat} (see Section \ref{subsec:lemmas}) leads to 
\begin{equation*}
\begin{split}
\E\left (W(\xi_{(i)},\xi_{(j)})-W(\xi_{(u)},\xi_{(v)})\right )^{2}&\leq C(n_0/n)^{2\alpha'}\leq C(1/k)^{2\alpha'}
\end{split}
\end{equation*}
where we have used that $k=\lfloor n/n_0\rfloor$. Plugging this bound into \eqref{prop_2_1}  proves the proposition.


\subsection{Proof of Proposition \ref{prp:bias_sbm}}

For any $W_0\in \cW[k]$, we first construct an ordered graphon $W'$  isomorphic to $W_0$ and we set $f'=\rho_nW'$. Then we construct an ordered empirical graphon $\widehat f'$ isomorphic to $\widetilde{f}_{\bTheta_0}$ and we estimate the $\delta(\cdot,\cdot)$-distance between these two ordered versions $f'$ and~$\widehat f'$.

 Consider the matrix ${\bTheta}_0'$ with entries $({\bTheta}_0')_{ij}=\rho_nW(\xi_i,\xi_j)$ for all $i,j$. As opposed to $\bTheta_0$, the diagonal entries of ${\bTheta}_0'$ are not constrained to be null. By the triangle inequality, we  get 
\beq\label{eq:agnostic_decomposition}
\E\left[\delta^2\left(\widetilde{f}_{\bTheta_0} , f_0\right)\right]\leq 2\E\left[\delta^2\left(\widetilde{f}_{\bTheta_0} , \widetilde{f}_{{\bTheta}_0'}\right)\right]+ 2\E\left[\delta^2\left(\widetilde{f}_{{\bTheta}_0'} , f_0\right)\right]\ .
\eeq
Since the entries of $\bTheta_0$ coincide with those of ${\bTheta}_0'$ outside the diagonal, the difference $\widetilde{f}_{\bTheta_0}- \widetilde{f}_{{\bTheta}_0'}$ is null outside of a set of measure $1/n$. Also, the entries of ${\bTheta}_0'$ are smaller than $\rho_n$. It follows that 
$ \E[\delta^2(\widetilde{f}_{\bTheta_0} , \widetilde{f}_{{\bTheta}_0'})]\leq \rho_n^2/n$. Hence, it suffices   to prove that 
$$\E[\delta^2(\widetilde{f}_{{\bTheta}_0'} , f_0)]\leq C\rho_n^2\sqrt{k/n}.$$
We prove this inequality by induction on $k$. The result is trivial for $k=1$ as $\delta^2\left(\widehat{f}_{\bTheta_0'} , f_0\right)=0$.
Fix some $k>1$ and assume that the result is valid for $\cW[k-1]$. Consider any  $W_0\in \cW[k]$ and let $\bQ\in \bbR^{k\times k}_{\text{sym}}$ and $\phi:[0,1]\to [k]$ be associated to $W_0$ as in definition \eqref{eq:def_step_function}. We assume w.l.o.g. that all the rows of $\bQ$ are distinct and that  $\lambda_a:= \lambda(\phi^{-1}(a))$ is positive for all $a\in[k]$, since otherwise $W_0$ belongs to $\cW[k-1]$. For any $b\in [k]$, define the cumulative distribution function 
 \begin{equation*}
 F_{\phi}(b)=\sum_{a=1}^{b} \lambda_a 
 \end{equation*} 
and set $F_{\phi}(0)=0$. For any $(a,b)\in [k]\times [k]$ define $\Pi_{ab}(\phi)=[F_{\phi}(a-1),F_{\phi}(a))\times [F_{\phi}(b-1),F_{\phi}(b))$
where $\mathds{1}_{A}(\cdot)$ denotes the indicator function of set $A$. Finally, we consider the ordered graphon  
\[W'(x,y)= \sum_{a=1}^k \sum_{b=1}^k \bQ_{ab} \mathds{1}_{\Pi_{ab}(\phi)}(x,y)\ . \]
Obviously, $f'=\rho_n W'$ is weakly isomorphic to $f_0=\rho_n W_0$. 
Let 
$$\widehat{\lambda}_a=\frac1n \sum_{i=1}^n \mathds{1}_{\{ \xi_i\in \phi^{-1}(a)\}}$$ 
be the (unobserved) empirical frequency of group $a$. Here, $\xi_1,\ldots ,\xi_n$
are the i.i.d. uniformly distributed random variables in the graphon model \eqref{sparse_grapnon_mod}. 

Note that the relations $\sum_{a=1}^k \lambda_a=\sum_{a=1}^k \hat\lambda_a=1$ imply
\begin{equation}\label{eq:lam}
\sum_{a: \lambda_a>\widehat{\lambda}_a} (\lambda_a-\widehat{\lambda}_a)=\sum_{a: \widehat{\lambda}_a> \lambda_a} (\hat \lambda_a-\lambda_a).
\end{equation}
Consider a function $\psi:[0,1]\rightarrow [k]$  such that: 
\begin{itemize}
 \item[(i)] $\psi(x)= a$ for all $a\in [k]$ and $x\in [F_{\phi}(a-1), F_{\phi}(a-1)+ \widehat{\lambda}_a\wedge \lambda_a)$,
 \item[(ii)] $\lambda(\psi^{-1}(a))= \widehat{\lambda}_a$ for all $a\in [k]$. 
\end{itemize}
Such  a function $\psi$ exists. Indeed, for each $a$ such that $\lambda_a>\widehat{\lambda}_a$, conditions (i) and (ii) are trivially satisfied if we take $\psi^{-1}(a)=[F_{\phi}(a-1), F_{\phi}(a-1)+ \widehat{\lambda}_a)$, and there is an interval of Lebesgue measure $\lambda_a-\widehat{\lambda}_a$ left non-assigned. Summing over all such $a$, we see that there is a union of intervals with Lebesgue measure  $m_+:=\sum_{a: \lambda_a>\widehat{\lambda}_a} (\lambda_a-\widehat{\lambda}_a)$ left non-assigned. On the other hand, for $a$ such that $\lambda_a<\widehat{\lambda}_a$, we must have $\psi(x)=a$  for $x\in [F_{\phi}(a-1), F_{\phi}(a-1)+ \lambda_a)$ to satisfy (i), while to meet condition (ii) we need additionally to assign $\psi(x)=a$ for $x$ on a set of Lebesgue measure $\hat \lambda_a-\lambda_a$.  Summing over all such $a$, we need additionally to find a set of Lebesgue measure  $m_-:=\sum_{a: \widehat{\lambda}_a> \lambda_a} (\lambda_a-\widehat{\lambda}_a)$ to make such assignments. But this set is readily available as a union of non-assigned intervals for all $a$ such that $\lambda_a>\widehat{\lambda}_a$ since $m_+=m_-$ by virtue of \eqref{eq:lam}.

Finally define the graphon $\widehat{f}'(x,y)= \bQ_{\psi(x),\psi(y)}$. Notice that in view of (ii) $\widehat{f}'$ is weakly isomorphic to the empirical graphon $\widetilde{f}_{\bTheta_0}$. Since $\delta(\cdot,\cdot)$ is a metric on the quotient space $\widetilde{\cW}$, 
\[
\delta^2(\widetilde{f}_{\Theta_0},f_0)= \delta^2(\widehat{f}',f')\leq \int_{[0,1]^2} |\widehat{f}'(x,y)- f'(x,y)|^2 \mathrm{d}x \mathrm{d}y\leq \rho_n^2 \int_{[0,1]^2}  \mathds{1}_{\{f'(x,y)\neq \widehat{f}'(x,y)\}}\mathrm{d}x\mathrm{d}y.
\]
The two functions $f_0(x,y)$ and $\widehat{f}(x,y)$ are equal except possibly the case when either $x$ or $y$ belongs to one of the intervals $[F_{\phi}(a-1)+ \widehat{\lambda}_a\wedge \lambda_a, F_{\phi}(a-1)+\lambda_a)$ for $a\in [k]$. Hence, the Lebesgue measure of the set $\{(x,y): f'(x,y)\neq \widehat{f}'(x,y)\}$ is not greater than $2 m_+ = m_+ + m_-= \sum_{a=1}^k |\lambda_a-\widehat{\lambda}_a|$. Thus, 
\[\delta^2(\widehat{f}_{\bTheta_0'},f_0)\leq \rho_n^2 \sum_{a=1}^k |\lambda_a-\widehat{\lambda}_a|. \]
Since $\xi_1,\ldots ,\xi_n$
are i.i.d. uniformly distributed random variables,  $n\widehat{\lambda}_a$ has a binomial distribution with parameters ($n$, $\lambda_a$). By the Cauchy-Schwarz inequality we get $\E[|\lambda_a-\widehat{\lambda}_a|]\leq \sqrt{\frac{\lambda_a(1-\lambda_a)}{n}}$. Applying again the Cauchy-Schwarz inequality, we conclude that 
\[\E\left[\delta^2(\widehat{f}_{\bTheta_0'},f_0)\right]\leq \frac{\rho_n^2}{\sqrt{n}} \sum_{a=1}^k \sqrt{\lambda_a}\leq \rho_n^2\sqrt{\frac{k}{n}}.\]




\subsection{Proof of Proposition \ref{prp:tographon}}  

Arguing as in the proof of Proposition \ref{prp:bias_sbm}, we have 
\[\E\left[\delta^2\left(\widetilde{f}_{\bTheta_0} , f_0\right)\right]\leq \frac{2\rho_n^2}{n}+   2\E\left[\delta^2\left(\widetilde{f}_{{\bTheta}_0'} , f_0\right)\right],\]
where we recall that ${\bTheta}_0'$ is defined by $({\bTheta}_0')_{ij}=\rho_nW(\xi_i,\xi_j)$ for all $i,j$. 
Hence, it suffices to prove that 
\[\E\left[\delta^2\left(\widetilde{f}_{{\bTheta}_0'} , f_0\right)\right]\leq
C\frac{\rho_n^{2}}{n^{\alpha\wedge 1}}\ . 
 \]

\medskip 

  We have 
\begin{equation*}
\begin{split}
\delta^2\left(\widetilde{f}_{{\bTheta}_0'} , f_{0}\right)=\underset{\tau\in \mathcal M}{\inf}\sum_{i,j=1}^{n}\int _{(j-1)/n}^{j/n}\int _{(i-1)/n}^{i/n}\vert f_0(\tau(x),\tau(y))-({\bTheta}_0' )_{ij}\vert^{2}\mathrm{d}x\mathrm{d}y.
\end{split}
\end{equation*} 
The infimum over all measure-preserving bijections is smaller than the minimum over the subclass of  measure-preserving bijections $\tau$ satisfying the following property
\[\int _{(j-1)/n}^{j/n}\int _{(i-1)/n}^{i/n}W(\tau(x),\tau(y))\mathrm{d}x\mathrm{d}y=
\int _{(\sigma(j)-1)/n}^{\sigma(j)/n}\int _{(\sigma(i)-1)/n}^{\sigma(i)/n}W(x,y)\mathrm{d}x\mathrm{d}y\]
for some permutation $\sigma = (\sigma(1),\dots,\sigma(n))$ of $\{1,\dots, n\}$. Such $\tau$ correspond to permutations of intervals $[(i-1)/n,i/n]$ in accordance with $\sigma$.  For $(x,y)\in [(\sigma(i)-1)/n,\sigma(i)/n]\times [(\sigma(j)-1)/n,\sigma(j)/n]$ we use the bound
\begin{eqnarray}\label{prop_L2_2}
&&\vert \rho_{n}W_{0}(x,y)-\rho_{n}W_{0}(\xi_i,\xi_j)\vert\leq \left \vert \rho_{n}W_{0}(x,y)-\rho_{n}W_{0}\left (\tfrac{\sigma(i)}{n+1},\tfrac{\sigma(j)}{n+1}\right )\right \vert \\
&&+\left \vert \rho_{n}W_{0}\left (\tfrac{\sigma(i)}{n+1},\tfrac{\sigma(j)}{n+1}\right ) - \rho_{n}W_{0}\left (\xi_{(\sigma(i))},\xi_{(\sigma(j))}\right)
\right  \vert +\left \vert  \rho_{n}W_{0}\left (\xi_{(\sigma(i))},\xi_{(\sigma(j))}\right)-\rho_{n}W_{0}(\xi_i,\xi_j)
\right  \vert \nonumber
\end{eqnarray}
where $\xi_{(m)}$ denotes the $m$th largest element of the set $\{\xi_{1},\dots ,\xi_{n}\}$. We choose a random permutation $\sigma = ( \sigma(1),\dots, \sigma(n))$ 
such that $\xi_{\sigma^{-1}(1)}\leq \xi_{\sigma^{-1}(2)}\leq \dots \leq \xi_{\sigma^{-1}(n)}$. With this choice of $\sigma$, we have that $\left (\xi_{(\sigma(i))},\xi_{(\sigma(j))}\right)=(\xi_i,\xi_j)$ and $\left \vert  \rho_{n}W_{0}\left (\xi_{(\sigma(i))},\xi_{(\sigma(j))}\right)-\rho_{n}W_{0}(\xi_i,\xi_j)
\right  \vert=0$ almost surely.

For the first summand in \eqref{prop_L2_2}, as $W_0(\cdot,\cdot)$ satisfies \eqref{lip} and $$\left (\tfrac{\sigma(i)}{n+1},\tfrac{\sigma(j)}{n+1}\right )\in [(\sigma(i)-1)/n,\sigma(i)/n]\times [(\sigma(j)-1)/n,\sigma(j)/n]$$ we get
\begin{eqnarray}\label{prop_L2_200}
&&\left \vert W_0(x,y)-W_0\left (\tfrac{\sigma(i)}{n+1},\tfrac{\sigma(j)}{n+1}\right )\right \vert\leq 2L n^{-\alpha'},
\end{eqnarray}
where we set for brevity $\alpha'=\alpha\wedge 1$. To evaluate the contribution  of the second summand on the right hand side of \eqref{prop_L2_2} we use that 
\begin{eqnarray}\label{prop_L2_21}
\left \vert W_0\left (\tfrac{\sigma(i)}{n+1},\tfrac{\sigma(j)}{n+1}\right ) - W_0\left (\xi_{(\sigma(i))},\xi_{(\sigma(j))}\right)
\right  \vert\leq M\left (\left \vert \tfrac{\sigma(i)}{n+1}-\xi_{(\sigma(i))}\right \vert^{\alpha'}
+\left \vert \tfrac{\sigma(j)}{n+1}-\xi_{(\sigma(j))}\right \vert^{\alpha'} \right ).
\end{eqnarray}
Squaring, integrating and taking expectation we obtain
\begin{eqnarray}
 &&  \E\left [\sum_{i,j}  \int _{(\sigma(j)-1)/n}^{\sigma(j)/n}\int _{(\sigma(i)-1)/n}^{\sigma(i)/n} \, \left \vert \tfrac{\sigma(i)}{n+1}-\xi_{(\sigma(i))}\right \vert^{2\alpha'}\mathrm{d}x\mathrm{d}y
\right]\nonumber
\\
&& =  \frac{1}{n} \E\left [\sum_{m=1}^{n}  \left \vert \tfrac{m}{n+1}-\xi_{(m)}\right \vert^{2\alpha'}
\right] \nonumber
\\ \label{prop_L2_22}
&&\le \max_{m=1,\dots,n} \E \left[  \left \vert \tfrac{m}{n+1}-\xi_{(m)}\right \vert^{2\alpha'}\right] \le \max_{m=1,\dots,n}\left(\Var(\xi_{(m)})\right )^{\alpha'}\leq Cn^{-\alpha'}
\end{eqnarray}
where we have used the relations $\E ( \xi_{(m)})= \tfrac{m}{n+1}$, $\Var(\xi_{(m)})\le C/n$, and Jensen's inequality. The contribution corresponding to the second summand on the right hand side of \eqref{prop_L2_21} is evaluated analogously.
Combining \eqref{prop_L2_200} -- \eqref{prop_L2_22} with  \eqref{prop_L2_2} 
we get
\begin{eqnarray}\label{prop_L2_3}
&&\E\left[\delta^2\left(\widetilde{f}_{{\bTheta}_0'} , f_{0}\right)\right]\leq C\rho_{n}^{2}n^{-\alpha'}.
\nonumber
\end{eqnarray}
\subsection{Proof of Corollary \ref{prp:empirical_smooth_raphon}}
To prove the first part of the corollary, notice that its assumptions imply that the partition is balanced and $\rho_n\ge C k\log(k)/n$. Thus, similarly to Corollary \ref{cor:1}, we obtain from Proposition \ref{prp:l2} that
$$
\E\left[\frac{1}{n^{2}}\|\widehat{\bTheta}- \bTheta_0\|_F^2\right]\leq \frac{C}{n^{2}} \|\bTheta_0-\bTheta_{*,n_0}\|_F^2 + C\rho_n\left (\frac{\log(k)}{n}+ \frac{k^2}{n^{2}}\right ).
$$
Using Proposition \ref{approx_sbm_balanced} to bound the expectation of the first summand on the right hand side  we get
\beq\label{eq:risk_empirical_smooth_2}
  \E\left[\dfrac{1}{n^{2}}\|\widehat{\bTheta} -\bTheta_0\|_F^2 \right] \leq C\left \{ \dfrac{\rho_{n}^2}{k^{2(\alpha\wedge 1)}} + 
  \rho_n\left (\frac{\log(k)}{n}+ \frac{k^2}{n^{2}}\right )
  \right\} .
   \eeq
 Now \eqref{eq:risk_empirical_smooth_1} follows from \eqref{eq:risk_empirical_smooth_2} by taking $k=\left \lceil \left (\rho_nn^{2}\right )^{\frac{1}{2(1+\alpha\wedge 1)}} \right \rceil$ 
 if $\rho_n\geq n^{(\alpha\wedge 1)-1}$  and $k=\Big\lceil \left (\rho_n n\right )^{\frac{1}{2(\alpha\wedge 1)}} \Big\rceil$ if $\rho_n< n^{(\alpha\wedge 1)-1}$.  
 Bound \eqref{eq:risk_empirical_smooth} for the restricted least squares estimator follows from Propositions \ref{prp:l2_Thres} and  \ref{prp:gao}.  

%
%

\subsection{Proof of Corollary \ref{cor:integrated_risk_sbm}}
To prove part (i) of the corollary, we control the size of each block of $\bTheta_0$.  For any $a$ in $[k]$, the number $N_a$ of nodes belonging to block $a$ is a binomial random variable with parameters $(n,1/k)$ since the graphon $W$ is balanced. By Bernstein's inequality,
\[\P\left[N_a-\frac{n}{k}\leq - t\right]\leq \exp\left[-\frac{t^2/2}{n/k+ t/3}\right]\ .\]
Taking $t=n/(2k)$ in the above inequality, we obtain
\[\P\left[N_a\leq \frac{n}{2k}\right]\leq \exp[- Cn/k].\]
This inequlity and the union bound over all $a\in [k]$ imply that the size of all blocks of  $\bTheta_0$ is greater than $n_0$ with probability greater than $1-k\exp(-Cn/k)$. Together with Propositions \ref{prp:l2} and \ref{prp:bias_sbm}, this yields
\[
\E\left[\delta^2\left(\widehat{f} , f_0\right)\right]\leq C'\left(  \rho_n \frac{k^2}{n^2}+  \rho_n \frac{\log(k)}{n}+ \rho_n^2  \sqrt{\frac{k}{n}}\right) + k\exp(-Cn/k) \ ,\]
where the last summand is negligible. The second part of the corollary is a straightforward consequence of  Propositions \ref{prp:l2} and \ref{prp:l2_Thres}. 

\subsection{Proof of Proposition \ref{prp:lower_minimax_integrated_risk}}\label{sec:proof_minimax_integrated_risk}

It is enough to prove  separately the following three  minimax lower bounds.
\begin{eqnarray} \label{eq:lower1}
 \inf_{\widehat{f}}\sup_{W_0\in \cW[k]}\E_{W_0}[\delta^2(\widehat{f}, \rho_n W_0)] &\geq& C \rho^2_n \sqrt{\frac{k-1}{n}}\ ,\\
 \inf_{\widehat{f}}\sup_{W_0\in \cW[k]}\E_{W_0}[\delta^2(\widehat{f}, \rho_n W_0)] &\geq& C 
\left(\rho_n \frac{k^2}{n^2}\wedge \rho_n^2\right),
 \label{eq:lower2} \\
 \inf_{\widehat{f}}\sup_{W_0\in \cW[2]}\E_{W_0}[\delta^2(\widehat{f}, \rho_n W_0)] &\geq& C  
\left( \frac{\rho_n}{n}\wedge \rho_n^2\right).
 \label{eq:lower3}
\end{eqnarray}

\subsubsection{Proof of \eqref{eq:lower1}}

Without loss of generality, it suffices to prove \eqref{eq:lower1}
for $k=2$ and for all $k=16{\bar k}$ with integer ${\bar k}$ large enough. Indeed, for any $k> 2$, there exists $k'\leq k$ such that $k-k'\leq 15$ and $k'$ is either a multiple of $16$ or equal to $2$, so that 
\[
 \inf_{\widehat{f}}\sup_{W_0\in \cW[k]}\E_{W_0}\left[\delta^2\left(\widehat{f}; f_0\right)\right]\geq\inf_{\widehat{f}}\sup_{W_0\in \cW[k']}\E_{W_0}\left[\delta^2\left(\widehat{f}; f_0\right)\right] \geq C  \rho_n^2  \sqrt{\frac{k'-1}{n}}\geq \frac{C}{4}  \rho_n^2  \sqrt{\frac{k-1}{n}}\ .
 \]
We  first consider $k$ which is a multiple of $16$. 
The case $k=2$ is sketched afterwards.

The proof follows the general scheme of reduction to testing of finite number of hypotheses (cf., e.g., \cite{tsybakov_book}). The main difficulty is to obtain the necessary lower bound for the distance $\delta(\cdot,\cdot)$ between the graphons generating the hypotheses since this distance is defined as a minimum over all measure-preserving bijections. We start by constructing a matrix $\bB$ from which the step-function graphons will be derived.

\begin{lemma}\label{lem:existence_hadamard_matrix}
Fix $\eta_0=1/16$ and assume that $k$ is a multiple of 16 and is greater than some  constant. 
There exists a $k\times k$ symmetric $\{-1,1\}$  matrix  $\bB$ satisfying the following two properties.
\begin{itemize}
 \item For any $(a,b)\in [k]$ with $a\neq b$, the inner product between the two columns $\langle \bB_{a,\cdot} , \bB_{b,\cdot}\rangle$ satisfies
 \beq\label{eq:property_1_B}
 |\langle \bB_{a,\cdot} , \bB_{b,\cdot}\rangle|\leq k/4.
 \eeq
 \item For any two subsets $X$ and $Y$ of $[k]$ satisfying $|X|=|Y|=\eta_0 k$ and $X\cap Y=\emptyset$ and any labelings $\pi_1:[\eta_0  k ]\to X$ and $\pi_{2}:[\eta_0 k ]\to Y$, we have
 \beq\label{eq:property_2_B}
 \sum_{a,b=1}^{ \eta_0 k}[\bB_{\pi_1(a),\pi_1(b)}-\bB_{\pi_2(a),\pi_2(b)} ]^2\geq  \eta_0^2 k^2/2.
 \eeq
\end{itemize}
\end{lemma}

Note that any Hadamard matrix satisfies condition \eqref{eq:property_1_B} since its columns are orthogonal. Unfortunately, the second condition \eqref{eq:property_2_B} seems difficult to check for such matrices. This is why we adopt a probabilistic approach in the proof of this lemma showing that, with positive probability, a symmetric matrix with independent Rademacher entries satisfies the above conditions.

The graphon ``hypotheses'' that we consider in this proof are generated by the matrix of connection probabilities 
 $\bQ : = (\bJ + \bB)/2$ where $\bB$ is a matrix from Lemma \ref{lem:existence_hadamard_matrix} and $\bJ$ is a $k\times k$ matrix with all entries equal to 1.

Fix some $\epsilon < 1/(4k)$. Denote by $\cC_0$ the collection of vectors $u\in \{-\epsilon,\epsilon\}^k$ satisfying $\sum_{a=1}^k u_a=0$. For any $u\in \cC_0$,  define the cumulative distribution $F_u$ on $\{0,\ldots,  k\}$ by the relations $F_u(0)=0$ and   $F_u(a)= a/k + \sum_{b=1}^a u_b$ for $a\in [k]$. Then, set $\Pi_{ab}(u)=[F_{u}(a-1),F_{u}(a))\times [F_{u}(b-1),F_{u}(b))$ and define the graphon $W_u\in \cW[k]$ by
 \begin{equation*}
 W_u(x,y)=\sum_{(a,b)\in [k]\times [k]}{\bQ}_{ab}\mathds{1}_{\Pi_{ab}(u)}(x,y).
 \end{equation*}
The graphon $W_u$ 
is slightly unbalanced as the weight of each class is either $1/k-\epsilon$ or $1/k+\epsilon$.  

Let $\P_{W_u}$ denote the distribution of observations $\bA':=(\bA_{ij}, 1\le j<i\le n)$ sampled according to the sparse graphon model \eqref{sparse_grapnon_mod} with $W_0=W_u$. Since the matrix $\bQ$ is fixed the difficulty in distinguishing between the distributions $\P_{W_u}$ and $\P_{W_v}$ for $u\ne v$ comes from the randomness of the design points $\xi_1,\ldots,\xi_n$ in the graphon model \eqref{sparse_grapnon_mod} rather than from the randomness of the realization of $\bA'$ conditionally on $\xi_1,\ldots,\xi_n$. The following lemma gives a bound on the Kullback-Leibler  divergences $\mathcal{K}(\P_{W_u},\P_{W_v})$
between  $\P_{W_u}$ and $\P_{W_v}$.

\begin{lemma}\label{lem:kullback}
 For all $u,v\in \cC_0$ we have
 $$
\mathcal{K}(\P_{ W_u},\P_{ W_v})\le 16n k^2 \epsilon^2/3.
 $$
\end{lemma}

 Next, we need the following combinatorial result in the spirit of the Varshamov-Gilbert lemma~\cite[Chapter 2]{tsybakov_book}. It shows that there exists a large subset of $\cC_0$ composed of vectors $u$ that are well separated in the Hamming distance.
We state this result in terms of the sets $\cA_u:=\{a\in [k]:\ u_a=\epsilon\}$ where $u\in \cC_0$. Notice that, by definition of $\cC_0$, we have~$|\cA_u|=k/2$ for all $u\in \cC_0$. 

\begin{lemma}\label{lem:varshamov_variation}
 There exists a subset $\cC$ of $\cC_0$ such that $\log |\cC|\geq k/16$ and 
 \beq
 |\cA_u\Delta \cA_v|> k/4\ 
 \eeq
 for any $u\neq v\in \cC$.
\end{lemma}

Lemmas~\ref{lem:varshamov_variation} and~\ref{lem:existence_hadamard_matrix} are used to obtain the following  lower bound on the distance $\delta(W_u,W_v)$ between two distinct graphons in $\cC$. This lemma is the main ingredient of the proof.

\begin{lemma}\label{lem:separated}
For all $u,v\in \cC$ such that $u\ne v$, the graphons $W_u$ and $W_v$ are well separated in the $\delta(\cdot,\cdot)$ distance:
\[\delta^2(W_u,W_v)\geq \eta_0^2 k\epsilon/2,
\]
so that
 \beq\label{separ}
 \delta^2(\rho_n W_u,\rho_n W_v)\geq \rho_n^2\eta_0^2 k\epsilon/2, \quad \forall \  u,v \in \cC: u\ne v.
  \eeq 
\end{lemma}

Now, choose $\epsilon$ such that $\epsilon^2=\frac{3}{(16)^3 nk}$. Then it follows from Lemmas \ref{lem:kullback} and \ref{lem:varshamov_variation} that
 \beq\label{kull}
\mathcal{K}(\P_{ W_u},\P_{W_v})\le \frac1{16}\log |\cC|, \quad \forall \ u, v\in \cC: u\ne v.
 \eeq
In view of Theorem 2.7 in \cite{tsybakov_book}, inequalities \eqref{separ} and \eqref{kull} imply that
\[\inf_{\widehat{f}}\sup_{W_0\in \cW[k]}\E_{W_0}[\delta^2(\widehat{f}, \rho_n W_0)] \geq C \rho_n^2 \sqrt{\frac{k}{n}}\]
where $C>0$ is an absolute constant. This completes the proof for the case when $k$ is a large enough multiple of 16.

\smallskip

Let now $k=2$. Then, we reduce the lower bound to the problem of testing two hypotheses. We consider the matrix $\bB=\left(\begin{array}{cc} 1  &   1 \\ 1 & -1 \end{array}\right)$ and the two corresponding graphons $W_{u_1}$ and $W_{u_2}$ with $u_1=(\epsilon,-\epsilon)$ and $u_2=(-\epsilon,\epsilon)$. Adapting the argument of Lemma \ref{lem:separated}, one can prove that $\delta^2(\rho_n W_{u_1},\rho _nW_{u_2})\geq \rho_n^2\epsilon$. Moreover, exactly as in Lemma \ref{lem:kullback}, the Kullback-Leibler divergence between $\P_{W_{u_1}}$ and $\P_{W_{u_2}}$ is bounded by $Cn\epsilon^2$.  Taking  $\epsilon$ of the order $n^{-1/2}$ and using Theorem 2.2 from \cite{tsybakov_book} we conclude that
\[\inf_{\widehat{f}}\sup_{W_0\in \cW[2]}\E_{W_0}[\delta^2(\widehat{f}, \rho_n W_0)] \geq C \rho_n^2 \sqrt{\frac{1}{n}}\]
where $C>0$ is an absolute constant.

\begin{proof}[Proof of Lemma \ref{lem:existence_hadamard_matrix}]
Let $\bB$ be a $k\times k$ symmetric random matrix such that  its elements $\bB_{a,b}$, $a\leq b\in [k]$, are independent Rademacher random variables. It suffices to prove that $\bB$ satisfies properties \eqref{eq:property_1_B} and \eqref{eq:property_2_B} with positive probability.
Fix $a\neq b$. Then,   $\langle \bB_{a,\cdot} , \bB_{b,\cdot}\rangle$ is distributed as a sum of $k$ independent Rademacher variables. By Hoeffindg's inequality
\[\P\left[|\langle \bB_{a,\cdot} , \bB_{b,\cdot}\rangle|\geq k/4 \right]\leq 2\exp[-k/32].\]
By the union bound, property \eqref{eq:property_1_B} is satisfied for all $a\neq b$ with probability greater than $1- 2k^2\exp[-k/32]$. For $k$ larger than some absolute constant, this probability is greater than $3/4$. 

Fix any two subsets $X$ and $Y$ of $[k]$ such that $|X|=|Y|=\eta_0k$ and $X\cap Y=\emptyset$. Let  $\pi_1$ and $\pi_2$ be any two labelings of $X$ and $Y$. Then, define $T_{\pi_1,\pi_2}:= \sum_{a,b=1}^{ \eta_0 k}[\bB_{\pi_1(a),\pi_1(b)}-\bB_{\pi_2(a),\pi_2(b)} ]^2$. By symmetry of $\bB$, $T_{\pi_1,\pi_2}/8$ is greater than 
$\sum_{a<b}[\bB_{\pi_1(a),\pi_1(b)}-\bB_{\pi_2(a),\pi_2(b)} ]^2/4$ (we have put aside the diagonal terms) where all the summands $[\bB_{\pi_1(a),\pi_1(b)}-\bB_{\pi_2(a),\pi_2(b)} ]^2/4$ are independent Bernoulli random variables with parameter $1/2$ since $\pi_1([\eta_0k])\cap \pi_2([\eta_0k])=\emptyset $. Thus, $T_{\pi_1,\pi_2}/8$ is stochastically greater than a binomial random variable with parameters $\eta_0k(\eta_0k-1)/2$ and $1/2$. Applying again Hoeffding's inequality, we find
\[\P\left[T_{\pi_1,\pi_2}/8\leq \frac{\eta_0^2k^2}{8}- \frac{\eta_0k}{4} \right]\leq \exp\left[-\frac{\eta_0^2k^2}{32}\right].\]
For $k$ large enough we have $\eta_0k/4\leq \eta_0^2k^2/16$ so that $T_{\pi_1,\pi_2}\geq \eta_0^2k^2/2$ with probability greater than $1-\exp(-\eta_0^2k^2/32)$. 
There are less than $k^{2\eta_0k}$ such maps $(\pi_1,\pi_2)$ so that property \eqref{eq:property_2_B} is satisfied with probability greater than $1-\exp(2\eta_0k\log(k) - \eta_0^2k^2/32)$. Again, this probability is greater than $3/4$ for $k$ large enough. Applying once again the union bound we find that properties \eqref{eq:property_1_B} and \eqref{eq:property_2_B} are satisfied with probability greater than $1/2$.

\end{proof}

\begin{proof}[Proof of Lemma \ref{lem:varshamov_variation}]
 
Let $\cC$ be a maximal subset $\cC_0$ of points $u$ such that the corresponding $\cA_u$ are $k/4$-separated with respect to the measure of symmetric difference distance. By maximality of $\cC$, the union  of all balls in the Hamming distance centered at $u\in \cC$ with radii $k/4$ covers $\cC_0$. Denoting by $\cB[u,k/4]$ such a ball, we obtain by a volumetric argument
\beq\label{eq:volumetric}
|\cB[u,k/4]|/ |\cC_0| \geq   |\cC|^{-1}\ .
\eeq
If  we endow $\cC_0$ with the uniform probability, then $\left |\cB[u,k/4]\right ||\cC_0|^{-1}$ is the probability to draw a point from $\cB[u,k/4]$. Note that a point $v$ is in $\cB[u,k/4]$ if $\left \vert A_v\setminus A_u\right \vert\leq k/8$. One can construct the set $A_v$ doing $k/2$ sampling without replacement from the set $[k]$. We call an $a\in [k]$ a success if $a\in A_u$. There are $k/2$ of them. Let $S_k$ denotes the number of success, then   $S_k$ follows an hypergeometric distribution with parameters $(k, k/2,k/2)$ and we have 
 \[\frac{|\cB[u,k/4]|}{|\cC_0|}=  \P[S_k \geq 3k/8 ] \ . \]
 Now we only have to bound the deviations of $S_k$. It follows from \cite[p.173]{aldous85} that $S_k$ has the same distribution as the random variable $\E[\eta|\cB]$ where $\eta$ is a binomial random variable with parameters $(k/2,1/2)$ and $\cB$ is some suitable $\sigma$-algebra. Thus, by a convexity argument, we obtain that, for any $\lambda>0$,
 \[\E e^{\lambda (S_k-k/4)}\leq \E e^{\lambda (\eta-k/4)}\leq e^{\lambda^2 k/16}\ ,  \]
 where the last bound is due to Hoeffding's inequality. Applying the exponential Markov inequality and choosing $\lambda=1$, we obtain
 \[P\left[S_k-\frac{k}{4}\geq k/8 \right]\leq \exp\left[-k/16\right]\ \]
 and, in view of \eqref{eq:volumetric}, we conclude that  $|\cC|\geq \exp[k/16]$.
\end{proof}

\begin{proof}[Proof of Lemma \ref{lem:separated}]

Let $u$ and  $v$ be two different vectors in $\cC$ and let $\tau$ be a measure preserving bijection $[0,1]\mapsto [0,1]$.  We aim to prove that, {for any $\tau$}, 
\beq\label{eq:lower_l2_graphon}
\int \vert W_u(x,y) - W_v(\tau(x),\tau(y))\vert ^2dx dy \geq \eta_0^2k\epsilon/2 \ , 
\eeq
where we recall that $\eta_0=1/16$.
\medskip

If $x$ and {$x'$} correspond to two different classes of $W_u$, that is $x\in [F_{u}(a-1),F_{u}(a))$ and $x'\in [F_{u}(b-1),F_{u}(b))$ for some $a\neq b$, then the inner product between $W_u(x,\cdot)$ and  $W_u(x',\cdot)$ satisfies 
\begin{eqnarray}
\Big|\int (W_u(x,y)-1/2)(W_u(x',y)-1/2)dy\Big|&= &  {\Big|}\frac{1}{4}\sum_{c=1}^k \left(\frac{1}{k}+{u_c}\right)\bB_{ac}\bB_{bc} {\Big|}\nonumber \\
& \leq & \frac{1}{4k} \langle \bB_{a,.},\bB_{b,.}\rangle + \frac{1}{4}k\epsilon\nonumber\\
&\leq& 1/8\label{eq:upper_inner_graphon} \ , 
\end{eqnarray}
since we assume that $4k\epsilon\leq 1$.

For any $a,b\in [k]$, define $\omega_{ab}$ the  Lebesgue measure of $[F_{u}(a-1), F_u(a)]\cap \tau([F_v(b-1),F_v(b)])$. Since $\tau$ is measure preserving, $\sum_b \omega_{ab}= 1/k+u_a$ and $\sum_a \omega_{ab}= 1/k+v_b$. For any $a$ and $b\in[k]$ define $h_{u,a}(y)= W_u(F_u(a-1),y)-1/2$ and  $k_{v,b}(y)= W_v(F_v(b-1),\tau(y))-1/2$. Then, we have
 \beqn
\int \vert W_u(x,y) - W_v(\tau (x),\tau (y))\vert ^2dx dy = \sum_{a=1}^k \sum_{b=1}^k \omega_{a,b}\int \vert h_{u,a}(y) - k_{v,b}(y)\vert^2dy. 
 \eeqn
 Let $\langle \cdot,\cdot \rangle_2$ and $\|\cdot\|_2$ denote the standard inner product and the Euclidean norm in $L_2([0,1])$. By definition of $W_{u}$ we have that $|k_{v,a}(y)|=1/2$   for all $y\in[0,1]$ and any $a\in[k]$ which implies $\|k_{v,a}\|_2=1/2$. Now for
$b_1\neq b_2$, $\|k_{v,b_1}-k_{v,b_2}\|_2^2\geq  1/2 - 1/4=1/8$ by \eqref{eq:upper_inner_graphon}. By the triangle inequality, this implies that  
\[\|h_{u,a}-k_{v,b_1}\|_2^2 + \|h_{u,a}-k_{v,b_2}\|_2^2\geq \frac{\|k_{v,b_1}-k_{v,b_2}\|_2^2}{2}\geq   1/16\ .  \] 
As a consequence, for any $a\in [k]$ there exists at most one $b\in [k]$ such that $\|h_{u,a}- k_{v,b}\|^2< 1/32$. If such an index $b$ exists, it is denoted by $\pi(a)$. 
Exactly the same argument shows that for any $b\in [k]$ there exists at most one $a\in [k]$ such that $\|h_{u,a}- k_{v,b}\|^2< 1/32$ which implies that
there exists no $a\neq a'$ such  that $\pi(a)= \pi(a')$. Thus, it is possible to extend $\pi$ to a permutation of $[k]$. We get
\beqn
\int \vert W_u(x,y) - W_v(\tau(x),\tau(y))\vert ^2dx dy &\geq &\frac{1}{32}\sum_{a=1}^k \sum_{b\neq \pi(a)} \omega_{a,b}= \frac{1}{32}\sum_{a=1}^k\left(1/k+ u_a -\omega_{a,\pi(a)}  \right). 
\eeqn
If the sum $\sum_{a=1}^k1/k+ u_a -\omega_{a,\pi(a)}$ is greater than $k\epsilon/16$, then \eqref{eq:lower_l2_graphon} is satisfied. Thus, we can assume in the sequel that  $\sum_{a=1}^k1/k+ u_a -\omega_{a,\pi(a)}\leq k\epsilon/16$.  Using that $\omega_{a,b}\leq (1/k+u_a)\wedge (1/k+v_b)$ and the cardinality of the collection $\{a\in[k]:\, u_a>0\}$ is $k/2$ we deduce  that the  collection $\{a\in[k]:\, u_a>0,\ v_{\pi(a)}>0\text{ and } \omega_{a,\pi(a)}\geq 1/k\}$ has cardinality greater than $7k/16$. Since for $u\neq v\in \cC$,  $|\cA_u\cap \cA_v|\leq 3k/8$ (Lemma \ref{lem:varshamov_variation}), there exist subsets $A\subset \cA_{u}$ and $B\subset \cA_{v}$ of cardinality $\eta_0k$ (recall that $\eta_0=1/16)$ such that $\pi(A)=B$, $A\cap B=\emptyset$, and $\omega_{a,\pi(a)}\geq 1/k$ for all $a\in A$. Hence, 
\beqn
\lefteqn{\int \vert W_u(x,y) - W_v(\tau(x),\tau(y))\vert ^2dx dy} &&\\&\geq & \sum_{a_1\in A}\sum_{a_2\in A}\int_{[F_u(a_1-1),F_u(a_1))\times [F_u(a_2-1 ),F_u(a_2))}\vert W_{u}(x,y) -W_v(\tau(x),\tau(y)\vert^2 dxdy
\\ 
&\geq & \sum_{a_1\in A}\sum_{a_2\in A} \omega_{a_1,\pi(a_1)}\omega_{a_2,\pi(a_2)} [\bQ_{a_1,a_2} - \bQ_{\pi(a_1),\pi(a_2)}]^2\\
&\geq &  \frac{1}{4k^2} \sum_{a_1\in A}\sum_{a_2\in A}  [\bB_{a_1,a_2} - \bB_{\pi(a_1),\pi(a_2)}]^2,
\eeqn 
where the last inequality follows from the facts that $\bQ=(\bJ+ \bB)/2$ and $\omega_{a,\pi(a)}\geq 1/k$. Finally, we apply the property \eqref{eq:property_2_B} of $\bB$ to conclude that 
\[\int \vert W_u(x,y) - W_v(\tau(x),\tau(y))\vert ^2dx dy \geq \eta_0^2/8\ge\eta_0^2k\epsilon/2. \]
This proves \eqref{eq:lower_l2_graphon} and thus the lemma.

\end{proof}

 \begin{proof}[Proof of Lemma \ref{lem:kullback}]
 For $u\in \cC_0$, let $\zeta(u)=(\zeta_1(u),\ldots ,\zeta_n(u))$ be the vector of $n$ i.i.d. random variables with the discrete distribution on $[k]$ defined by $\P[\zeta_1(u)=a]= 1/k + u_a$ for any $a\in [k]$. Let $\bTheta_0$ be the $n\times n$ symmetric matrix with elements $(\bTheta_0)_{ii}=0$ and $(\bTheta_0)_{ij}=\rho_n \bQ_{\zeta_i(u),\zeta_j(u)}$ for $i\neq j$. Assume that, conditionally on $\zeta(u)$, the adjacency matrix $\bA$ is sampled according to the network sequence model with such probability matrix $\bTheta_0$. Notice that in this case the observations $\bA'=(\bA_{ij}, 1\le j<i\le n)$ have the probability distribution $\P_{W_u}$.  
 Using this remark and introducing the probabilities $\alpha_{\boldsymbol{a}} (u) = \P[\zeta(u)=\boldsymbol{a}] $ and $p_{A\boldsymbol{a}}=\P[\bA'=A\vert \zeta(u)=\boldsymbol{a}]$
for $\boldsymbol{a}\in [k]^n$, we can write the Kullback-Leibler divergence between $\P_{W_u}$ and $\P_{W_v}$ in the form
$$
\mathcal{K}(\P_{ W_u},\P_{ W_v})=\sum_{A}\sum_{\boldsymbol{a}}p_{A\boldsymbol{a}}\alpha_{\boldsymbol{a}} (u) \log\left(\frac{\sum_{\boldsymbol{a}}p_{A\boldsymbol{a}}\alpha_{\boldsymbol{a}} (u)}{\sum_{\boldsymbol{a}}p_{A\boldsymbol{a}}\alpha_{\boldsymbol{a}} (v)}\right)
$$
where the sums in $\boldsymbol{a}$ are over $[k]^n$ and the sum in $A$ is over all triangular upper halves of matrices in $ \{0,1\}^{n\times n }$.   Since 
the function $(x,y) \mapsto x\log(x/y)$ is convex we can apply Jensen's inequality to get
\begin{equation}\label{kullb_1}
\mathcal{K}(\P_{ W_u},\P_{ W_v}) \le \sum_{\boldsymbol{a}}\alpha_{\boldsymbol{a}} (u) \log\left(\frac{\alpha_{\boldsymbol{a}} (u)}{\alpha_{\boldsymbol{a}} (v)}\right)=
n\sum_{a\in[k]} (1/k + u_a)\log\left(\frac{1/k+u_a}{1/k+v_a}\right)
 \end{equation}
where the last equality follows from the fact that $\alpha_{\boldsymbol{a}} (u)$ are $n$-product probabilities. 
Since the Kullback-Leibler divergence does not exceed the chi-square divergence we obtain  
   $$
\sum_{a\in[k]} (1/k + u_a)\log\left(\frac{1/k+u_a}{1/k+v_a}\right) \le
\sum_{a\in[k]} \frac{(u_a-v_a)^2}{1/k+v_a}\le 16k^2\epsilon^2/3,
 $$
where last inequality uses that $|v_a|\le \epsilon \leq 1/(4k)$, and $|u_a-v_a|\le 2\epsilon$. Combining this with \eqref{kullb_1} proves the lemma. 
 \end{proof}

 \subsubsection{Proof of \eqref{eq:lower2}}
 As in the proof of Proposition  \ref{lower_sparse_sbm}, we use here Fano's method. The main difference is that the graphon separation in $\delta(\cdot,\cdot)$ distance is more difficult to handle than the matrix separation in the Frobenius distance.

First, note that it is sufficient to prove  \eqref{eq:lower2} for $k> k_0$ where $k_0$ is any fixed integer. Indeed, if $k\le k_0$,  the lower bound \eqref{eq:lower2} immediately follows from  \eqref{eq:lower3}.

 Let $\cC_0$ denote the set of all symmetric $k\times k$ matrices with entries in $\{-1,1\}$.  The graphon hypotheses that we consider in this proof are generated by matrices of intra-class connection probabilities of the form $\bQ_{\bB}:= (\bJ + \epsilon \bB)/2$ where $\bB\in \cC_0$,  $\epsilon\in (0,1/2)$, and $\bJ$ is a $k\times k$ matrix with all entries equal to $1$. Given a matrix $\bQ_{\bB}$, define the graphon 
 $W_{\bB}\in \cW[k]$ by the formula
 \[W_{\bB}(x,y) = \sum_{(a,b)\in [k]\times [k]}(\bQ_{\bB})_{ab}\mathbf{1}_{[(a-1)/k,a/k)\times [(b-1)/k, b/k)}(x,y).
 \]
As in the previous proof, we use the following combinatorial result in the spirit of the Varshamov-Gilbert lemma~\cite[Chapter 2]{tsybakov_book} that grants the existence of a large subset $\cC$ of $\cC_0$ such that the matrices  $\bB\in \cC$ are well separated in some sense.
Given any two permutations  $\pi$ and $\pi'$ of $[k]$ and any matrix $\bB$, we denote by $\bB^{\pi,\pi'}$ a matrix with entries $\bB^{\pi,\pi'}_{ab}= \bB_{\pi(a)\pi'(b)}$. 
 \begin{lemma}\label{lem:varshamov_varation2}
 Let $k_0$ be an integer large enough. For any  $k>k_0$,
 there exists a subset $\cC$ of $\cC_0$ satisfying $\log |\cC|\geq k^2/32$ and such that
 \beq\label{eq:distance_permutation}
 \|\bB_1-\bB_2^{\pi,\pi'}\|_F^2\geq \frac{k^2}{2}
 \eeq
 for all permutations $\pi$, $\pi'$  and all $\bB_1\neq \bB_2\in \cC$.
 \end{lemma}
We assume in the rest of this proof that $k$ is greater than $k_0$. As noticed above, it is enough to prove  \eqref{eq:lower2} only in this case. We choose a maximal subset $\cC$ satisfying the properties stated in Lemma~\ref{lem:varshamov_varation2}.
 The next lemma shows that the separation between matrices $\bB$ in $\cC$ translates into separation between the corresponding graphons $W_{\bB}$.
 \begin{lemma}\label{lem:lower_bound_distance_graphon_fano}
 Let $\cC$ be a maximal set  satisfying the properties stated in Lemma~\ref{lem:varshamov_varation2}. For any two distinct $\bB_1$ and $\bB_2$ in $\cC$, we have 
\beq\label{separ2}
\delta^2( \rho_n W_{\bB_1},\rho_n W_{\bB_2}) \geq \rho_n^2\epsilon^2/{8} .
\eeq
 \end{lemma}
 Finally, in order to apply Fano's method, we need to have an upper bound on the Kullback-Leibler divergence between the distributions $\P_{W_{\bB}}$ for $\bB\in \cC$.  It is given in the next lemma.
 
 \begin{lemma}\label{lem:kullback_2}
 For any $\bB_1$ and $\bB_2$ in $\cC$, we have
 \[
  \cK(\P_{W_{\bB_1}},\P_{W_{\bB_2}})\leq   {3}n^2\rho_n \epsilon^2\ .
 \]
 \end{lemma}

Choosing now $\epsilon=\left(\frac{k}{n\sqrt{\rho_n}}\wedge 1\right)/32$ we deduce from Lemmas \ref{lem:kullback_2} and \ref{lem:varshamov_varation2} that
\beq\label{kull2}
\mathcal{K}(\P_{W_{\bB_1}},\P_{ W_{\bB_2}})\le \frac1{10}\log |\cC|, \quad \forall \ \bB_1, \bB_2\in \cC: \bB_1\ne \bB_2.
\eeq 
In view of Theorem 2.7 in \cite{tsybakov_book}, inequalities \eqref{separ2} and \eqref{kull2} imply that
\[
\inf_{\widehat{f}}\sup_{W_0\in \cW[k]}\E_{W_0}[\delta^2(\widehat{f}, \rho_n W_0)] \geq C 
\left(\rho_n \frac{k^2}{n^2}\wedge \rho_n^2\right)
\]
where $C>0$ is an absolute constant. This completes the proof of \eqref{eq:lower2}.

 \begin{proof}[Proof of Lemma \ref{lem:varshamov_varation2}]

Define the pseudo-distance $d_{\Pi}$ on $\cC_0$ by $d_{\Pi}(\bB_1,\bB_2):= \underset{\pi,\pi'}{\min}\|\bB_1-\bB_2^{\pi,\pi'}\|_F$, where the minimum is taken over all permutations of $[k]$.
Let $\cC$ be a maximal subset $\cC_0$ of matrices $\bB$ that are $\sqrt{k(k+1)/2}$-separated with respect to $d_{\Pi}$. By maximality of $\cC$, the union  of all balls in $d_{\Pi}$ distance centered at $\bB\in \cC$ with radii $\sqrt{k(k+1)/2}$ covers $\cC_0$. Denoting by $\cB[\bB,\sqrt{k(k+1)/2},d_{\Pi}]$ such a ball, we obtain by a volumetric argument
\[
\left |\cB[\bB,\sqrt{k(k+1)/2},d_{\Pi}]\right |/ |\cC_0| \geq   |\cC|^{-1}\ .
\]
Since $d_{\Pi}(\bB_1,\bB_2)$ is the infimum over all $k!^2$ permutations of the distance $\|\bB_1-\bB_2^{\pi,\pi'}\|_F$, we have 
\[\left |\cB[\bB,\sqrt{k(k+1)/2},d_{\Pi}]\right |\leq k!^2 \left |\cB[\bB,\sqrt{k(k+1)/2},\|.\|_F]\right |\ ,\] where $\cB[\bB,\sqrt{k(k+1)/2},\|\cdot\|_F]$ is the ball centered at $\bB$ with respect to the Frobenius distance.  Given $\bB_1$ and $\bB_2$ in $\cC_0$, $\|\bB_1-\bB_2\|_F^2/4$ is equal to  the Hamming distance between $\bB_1$ and $\bB_2$. Consider the ball $\cB[\bB, k(k+1)/8, d_H]$ centered at $\bB$ with respect to the Hamming distance. As matrices in  $\cC_0$ are symmetric, $\cC_0$ is in bijection with $\{0,1\}^{k(k+1)/2}$. Using Hamming bound and Varshamov-Gilbert lemma~\cite[Chapter 2]{tsybakov_book} we obtain
\[
2^{-k(k+1)/2}|\cB[\bB,  k(k+1)/8 ,d_H]|\leq \vert C'\vert^{-1}\leq e^{-k(k+1)/16}\ 
\]
where $C'$ is a maximal subset of matrices $\bB$ that are $\frac{k(k+1)}{4}$-separated with respect to the Hamming distance.
We then conclude that
\[
 |\cC|\geq \frac{|\cC_0|}{k!^2|\cB[\bB,\sqrt{k(k+1)/2},\|\cdot\|_F]|}\geq \exp\left[k\{(k+1)/16-2\log(k)\}\right]\ ,
\]
which is larger $\exp(k^2/32)$ for $k$ large enough.

 \end{proof}

\begin{proof}[Proof of Lemma \ref{lem:lower_bound_distance_graphon_fano}]

Let $\bB_1$ and  $\bB_2$ be two distinct matrices in $\cC$ and let $\tau$ be a measure preserving bijection $[0,1]\mapsto [0,1]$.  Our aim is to prove that, {for any such $\tau$}, 
\beq\label{eq:lower_l2_graphon2}
\int \vert W_{\bB_1}(x,y) - W_{\bB_2}(\tau(x),\tau(y))\vert ^2dx dy \geq  \epsilon^2/8  .
\eeq
 For any $a,b\in [k]$, let $\omega_{ab}$ denote the  Lebesgue measure of $[(a-1)/k, a/k]\cap \tau([(b-1)/k,b/k])$. Since $\tau$ is measure preserving, $\sum_b \omega_{ab}= 1/k$ and $\sum_a \omega_{ab}= 1/k$. Hence, the $k\times k$ matrix $k\omega$, where $\omega = (\omega_{ab})_{a,b\in [k]}$, is doubly stochastic. For any permutation $\pi$ of $[k]$, denote $H(\pi)$ the corresponding permutation matrix.  By the Birkhoff--von Neumann theorem \cite{chvatal}, $k\omega$ is a convex combination of permutation matrices, that is there exist positive  numbers $\gamma_{\pi}$  such that $\sum_{\pi} \gamma_{\pi} =1/k$, and  $\omega= \sum_{\pi} \gamma_{\pi} H(\pi)$ where the summation runs over all permutations. Using these remarks we obtain
\beqn
\lefteqn{\int \vert W_{\bB_1}(x,y) - W_{\bB_2}(\tau(x),\tau(y))\vert ^2dx dy }&&
\\&=& \frac{\epsilon^2}{4} \sum_{a_1,a_2,b_1,b_2\in [k]}\omega_{a_1a_2}\omega_{b_1b_2}\big[(\bB_1)_{a_1b_1}-(\bB_2)_{a_2b_2}\big]^2\\
& =&\frac{\epsilon^2}{4} \sum_{\pi_1,\pi_2}  \sum_{a_1,a_2,b_1,b_2\in [k]} \gamma_{\pi_1}\gamma_{\pi_2} H(\pi_1)_{a_1a_2}H(\pi_2)_{b_1b_2}\big[(\bB_1)_{a_1b_1}-(\bB_2)_{a_2b_2}\big]^2\\&& (\text{by the definition of permutation matrices})
\\
& =& \frac{\epsilon^2}{4}\sum_{\pi_1,\pi_2}\gamma_{\pi_1}\gamma_{\pi_2}  \sum_{a,b}\big[(\bB_1)_{ab}-(\bB_2)_{\pi_1(a)\pi_2(b)}\big]^2\\
& = &\frac{\epsilon^2}{4}\sum_{\pi_1,\pi_2}\gamma_{\pi_1}\gamma_{\pi_2} \|\bB_{1}-\bB^{\pi_1,\pi_2}_2\|_F^2\\
&\geq & \sum_{\pi_1,\pi_2}\gamma_{\pi_1}\gamma_{\pi_2} \frac{\epsilon^2 k^2}{8}= \frac{\epsilon^2}{8} \ , \\
\eeqn
where we have used Lemma \ref{lem:varshamov_varation2} and the property  $\sum_{\pi}\gamma_{\pi}=1/k$. 

\end{proof}

 \begin{proof}[Proof of Lemma \ref{lem:kullback_2}] 
 The proof is quite similar to that of Lemma \ref{lem:kullback}. Fix two matrices $\bB_1$ and $\bB_2$ in $\cC$. Let $\zeta= (\zeta_1,\ldots,\zeta_n)$ be the vector of $n$ i.i.d. random variables with uniform distribution on $[k]$.
 Let $ \bTheta_1$ be the $n\times n$ symmetric matrix with elements $(\bTheta_1)_{ii}=0$ and $(\bTheta_1)_{ij}=\rho_n[1+ (\bB_1)_{\zeta_i,\zeta_j}]/2$ for $i\neq j$. Assume that, conditionally on $\zeta(u)$, the adjacency matrix $\bA$ is sampled according to the network sequence model with such probability matrix $\bTheta_1$. Notice that in this case the observations $\bA'=(\bA_{ij}, 1\le j<i\le n)$ have the probability distribution $\P_{W_{\bB_1}}$.  
Using this remark, we introduce the probabilities $\alpha_{\boldsymbol{a}} = \P[\zeta=\boldsymbol{a}] $ and $p^{(1)}_{A\boldsymbol{a}}=\P[\bA'=A\vert \zeta=\boldsymbol{a}]$
for $\boldsymbol{a}\in [k]^n$. Next, we introduce the analogous probabilities $p^{(2)}_{A\boldsymbol{a}}$ for a matrix $\bTheta_2$ depending on $\bB_2$ in the same way as $\bTheta_1$ depends on $\bB_1$.
The Kullback-Leibler divergence between $\P_{W_{\bB_1}}$ and $\P_{W_{\bB_2}}$ has the form
$$
\mathcal{K}(\P_{W_{\bB_1}},\P_{W_{\bB_2}})=\sum_{A}\sum_{\boldsymbol{a}}\alpha_{\boldsymbol{a}}p^{(1)}_{A\boldsymbol{a}}  \log\left(\frac{\sum_{\boldsymbol{a}}
\alpha_{\boldsymbol{a}}
p^{(1)}_{A\boldsymbol{a}}}{\sum_{\boldsymbol{a}} 
\alpha_{\boldsymbol{a}}
p^{(2)}_{A\boldsymbol{a}}}\right)
$$
where the sums in $\boldsymbol{a}$ are over $[k]^n$ and the sum in $A$ is over all triangular upper halves of matrices in $ \{0,1\}^{n\times n }$.   Since 
the function $(x,y) \mapsto x\log(x/y)$ is convex we can apply Jensen's inequality to get
\[
\mathcal{K}(\P_{ W_{\bB_1}},\P_{W_{\bB_2}}) \le \sum_{\boldsymbol{a}}\alpha_{\boldsymbol{a}} \sum_{A} p^{(1)}_{A\boldsymbol{a}} \log\left(\frac{p^{(1)}_{A\boldsymbol{a}}}{p^{(2)}_{A\boldsymbol{a}}}\right).
\]
Here, the sum in $A$ for fixed $\boldsymbol{a}$ is the Kullback-Leibler divergence between two $n(n-1)/2$-products of Bernoulli measures, each of which has success probability either $\rho_n(1+\epsilon)/2$ or $\rho_n(1-\epsilon)/2$.
Thus,  for $p=(1+\epsilon)/2$ and $q=(1-\epsilon)/2$ we have
   \begin{equation}\label{productt}
   \mathcal{K}(\P_{W_{\bB_1}},\P_{W_{\bB_2}}) \le \frac{n(n-1)}{2}\rho_n  \kappa(p,q),
 \end{equation}
 where $\kappa(p,q)$ is the Kullback-Leibler divergence between the Bernoulli measures with success probabilities $p$ and $q$. Since the Kullback-Leibler divergence does not exceed the chi-square divergence we obtain $\kappa(p,q)\le (p-q)^2(p^{-1}+q^{-1})= 4\epsilon^2/(1-\epsilon^2)\le 16\epsilon^2/3$ for any $\epsilon<1/2$. The lemma now follows by substitution of this bound on $\kappa(p,q)$ into \eqref{productt}.
 \end{proof}

  \subsubsection{Proof of \eqref{eq:lower3}}
 We use here a reduction to the problem of testing two simple hypotheses by Le Cam's method.  Fix some $0<\epsilon\le 1/4$. Let $W_1$ be the constant graphon with $W_1(x,y)\equiv 1/2$, and let $W_2\in \cW[2]$ be the $2$-step graphon with 
 $W_{2}(x,y)=1/2+\epsilon$ if $x,y\in [0,1/2)^2\cup [1/2,1]^2$ and $W_{2}(x,y)=1/2-\epsilon$ elsewhere. Obviously, we have
\beq \label{delta_distance_W1W2}
  \delta^2(\rho_n W_1,\rho_n W_2)= \rho_n^2 \epsilon^2 .
\eeq
We have
\begin{eqnarray}\nonumber 
 \inf_{\widehat{f}}\max_{W_0 \in \{W_1,W_2\}}\E_{W_0}[\delta^2(\widehat{f}, \rho_n W_0)] 
 &\geq& 
 \frac12\int\left(\delta^2(\widehat{f}, \rho_n W_1) + \delta^2(\widehat{f}, \rho_n W_2)\right)\min(d\P_{ W_1}, d\P_{ W_2})\\
 &\geq&
 \frac{\delta^2(\rho_n W_1,\rho_n W_2)}{4}
 \int\min(d\P_{ W_1}, d\P_{ W_2}) \nonumber\\
 &\ge& \frac{\rho_n^2 \epsilon^2}{8}\exp\left(- \chi^2(\P_{ W_2},\P_{ W_1})\right)
  \label{eq:lower_risk_3}
 \end{eqnarray}
where $\chi^2(\P_{ W_2},\P_{ W_1})$ is the chi-square divergence between $\P_{ W_2}$ and $\P_{ W_1}$. In the last inequality we have used (2.24) and (2.26) from \cite{tsybakov_book}, and \eqref{delta_distance_W1W2}. Finally, the following lemma allows us to conclude the proof by setting $\epsilon= \sqrt{\frac{c_0}{n\rho_n}}\wedge \frac{1}{4}$.

 \begin{lemma}\label{lem:upper_total_variation_distance}
There exists an absolute constant $c_0>0$ such that $\chi^2(\P_{ W_2},\P_{ W_1})\leq 1/4$ if  $\epsilon$ satisfies $n\rho_n \epsilon^2 \leq  c_0$.
 \end{lemma}

\begin{proof}[Proof of Lemma \ref{lem:upper_total_variation_distance}]

Let $L(\bA')$ be likelihood ratio of $\P_{ W_2}$ with respect to $\P_{ W_1}$. Since $\chi^2(\P_{ W_2},\P_{ W_1})=\E_{ W_1}[L(\bA')^2]-1$ it remains to prove that $\E_{ W_1}[L(\bA')^2]\leq 5/4$. 

For the sake of brevity, we write in what follows $\E[\cdot]=\E_{ W_1}[\cdot]$.
We also set $p_0:=\rho_n/2$,  $p_1:=\rho_n(1/2+\epsilon)$ and $p_2:=\rho_n(1/2-\epsilon)$.  

As the graphon $\rho_n W_2$ is a 2-step function, we may assume that the components of $\boldsymbol{\xi}=(\xi_1,\ldots, \xi_n)$ are i.i.d. uniformly distributed on $\{0,1\}$.
Given $\boldsymbol{\xi}$, define the collection $S:= \{\{a,b\}:\ \xi_a=\xi_b\}$ of subsets of indices with identical position.
For $\{i,j\}$ in $S$ (resp. $S^c$), $\bA_{i,j}$ follows a Bernoulli distribution with parameter $p_1$ (resp. $p_2$).
Denote $\mu$ the  distribution of $S$ and  $n^{(2)}=n(n-1)/2$  the number of subsets of size $2$ of $[n]$. Then, the likelihood has the form
\beqn 
 L(\bA')& =& \int L_S(\bA')d\mu(S)\ ,\\
 L_S(\bA')&:=& \left(\frac{1-p_1}{1-p_0}\right)^{|S|}\left(\frac{1-p_2}{1-p_0}\right)^{n^{(2)}-|S|}  \prod_{\{a,b\}\in S }\left(\frac{p_1(1-p_0)}{p_0(1-p_1)}\right)^{\bA_{ab}}\prod_{\{a,b\}\in S^c }\left(\frac{p_2(1-p_0)}{p_0(1-p_2)}\right)^{\bA_{ab}}\ .
\eeqn 
By Fubini theorem, we may write $\E[L^2(\bA')] = \int \mathbb{E}\left[L_{S_1}(\bA')L_{S_2}(\bA')\right]d\mu(S_1)d\mu(S_2)$, where 
\beqn 
\E\left[L_{S_1}(\bA')L_{S_2}(\bA')\right] &= & \left(\frac{1-p_1}{1-p_0}\right)^{|S_1|+|S_2|}\left(\frac{1-p_2}{1-p_0}\right)^{2n^{(2)}-|S_1|-|S_2|}
\E\Big[\prod_{\{a,b\}\in S_1\cap S_2 }\left(\tfrac{p_1(1-p_0)}{p_0(1-p_1)}\right)^{2\bA_{ab}}\times \\
&&
\prod_{\{a,b\}\in 
S_1^c\cap S_2^c }\left(\tfrac{p_2(1-p_0)}{p_0(1-p_2)}\right)^{2\bA_{ab}}
\prod_{\{a,b\}\in S_1\Delta S_2 }\left(\tfrac{p_1p_2(1-p_0)^{2}}{p_0^2(1-p_1)(1-p_2)}\right)^{\bA_{ab}}
\Big]
\\
& =& \left[1+ \tfrac{(p_1-p_0)^2}{p_0(1-p_0)}\right]^{|S_1\cap S_2|}\left[1+ \tfrac{(p_2-p_0)^2}{p_0(1-p_0)}\right]^{|S^c_1\cap S^c_2|}\left[1+ \tfrac{p_1p_2+p_0^2-p_1p_0-p_2p_0}{p_0(1-p_0)}\right]^{|S_1\Delta S_2|}\ . 
\eeqn 
Using the definition of $p_0$, $p_1$ and $p_2$, we find  
\beqn 
\E\left[L_{S_1}(\bA')L_{S_2}(\bA')\right]& = & \left[1+ \frac{\rho_n^2\epsilon^2}{p_0(1-p_0)}\right]^{|S_1\cap S_2|+|S^c_1\cap S^c_2|}\left[1- \frac{\rho_n^2\epsilon^2}{p_0(1-p_0)}\right]^{|S_1\Delta S_2|}\\
&\leq & \left[1+ \frac{\rho_n^2\epsilon^2}{p_0(1-p_0)}\right]^{|S_1\cap S_2|+|S^c_1\cap S^c_2|-|S_1\Delta S_2|}\\
& \leq & \exp\left[\left(2|S_1\cap S_2|+2|S^c_1\cap S^c_2| -n^{(2)} \right) 4\rho_n\epsilon^2 \right].
\eeqn 
Thus, to bound the second moment of $L(\bA')$, it suffices to control an exponential moment of $T:= |S_1\cap S_2|+|S^c_1\cap S^c_2|$ where $S_1$ and $S_2$ are independent and distributed  according to $\mu$. To handle this quantity, we denote by $\boldsymbol{\xi}=(\xi_1,\ldots, \xi_n)$ and by $\boldsymbol{\xi}'= (\xi'_1,\ldots, \xi'_n)$ the positions for the first and second sample corresponding to $S_1$ and $S_2$, respectively. Next, for any $i,j\in \{0,1\}^2$, define
\[N_{ij}:= |\big\{a: \xi_a=i\text{ and }{\xi'}_a=j\big\}|\ .\]
Then we have $2|S_1\cap S_2|+n=N_{00}^2+ N_{01}^2+ N_{10}^2+N^2_{11}$ and 
 $2|S_1^c\cap S_2^c|= 2N_{00}N_{11}+ 2N_{01}N_{10}$ so that $2T+n = (N_{00}+N_{11})^2 + (N_{01}+N_{10})^2$. 
Define the random variable $Z:= N_{00}+N_{11}- n/2$. It has a centered binomial distribution with parameters $n$ and $1/2$. We have
 \[
  2T- n^{(2)}= (n/2+Z)^2 + (n/2-Z)^2  -n-n^{(2)}  = 2Z^2 - n/2\ .
 \]
Plugging this identity into the expression for $\E[L^2(\bA')]$, we conclude that
\[
\E[L^2(\bA')]\leq \mathbf{E}\left[\exp\big(8\rho_n \epsilon^2 Z^2\big)\right]\ ,
\]
where $\mathbf{E}[\cdot]$ stands for the expectation with respect to the distribution of $Z$.
By Hoeffding's inequality, $Z$ has a subgaussian distribution with subgaussian norm smaller than $n^{1/2}$. Consequently, $\E[L^2(\bA')]\leq 5/4$ as soon as  $8\rho_n  n \epsilon^2$ is smaller than some numerical constant.
\end{proof}

\subsection{Technical lemmas}\label{subsec:lemmas}
\begin{lemma}\label{lemma_diff_ordered_stat}
Let $X_{1},\dots,X_{n}$ be i.i.d. uniformly distributed variables on $[0,1]$ and  $X_{(i)}$ is the $i$th element of the ordered sample $X_{(1)}\leq X_{(2)}\leq \dots \leq X_{(n)}$. Then, for any $n_0\leq n$, and $0\leq s<n_0$,
\[\E\left (X_{(i)}-X_{(i+s)}\right )^{2}=\dfrac{s(s+1)}{(n+1)(n+2)}\leq \left (\dfrac{n_0}{n}\right )^{2}\ .\]
\end{lemma}
\begin{proof}
Note that $X_{(i)}-X_{(i+s)}\sim \mathrm{Beta}(s,n-s+1)$. For $Y\sim \mathrm{Beta}(\beta,\gamma)$ we have 
\[\E (Y^{2})=\dfrac{\beta(\beta+1)}{(\beta+\gamma+1)(\beta+\gamma)}\]
which implies the lemma.
\end{proof}

\begin{proof}[Proof of Lemma \ref{lem:covering}]
Since $\|\hat{\bT}\|_{\infty}\leq 2r$, the definition of $q_{\max}$ and the fact that  $\| \hat{\bT}\|_F \ge r$ imply $\epsilon_0\leq \|\hat{\bT}\|_{F}\leq \epsilon_{q_{\max}}$. 
Denote by $\hat{q}$ the integer such that  $2^{-1}\epsilon_{\hat{q}}\leq \|\hat{\bT}\|_F< \epsilon_{\hat{q}}$.
Let $\hat{\bV}_0$ be a matrix minimizing $\|\hat{\bT}/\|\hat{\bT}\|_F-\bV\|_F$ over all $\bV$ in $\cC_{\hat{z}_r}$. Then, take $\hat{\bV}=\hat{\bV}_0^{\hat{q},\hat\bU,{\hat z}_r}$ where $\hat\bU$ is a minimizer of $\|\hat{\bT}- \hat{\bV}_0^{\hat{q},\bU,{\hat z}_r}\|_F$ over all $\bU\in \{-1,0,1\}^{k\times k}$. Notice that $\hat\bU$ is also a minimizer of $\|\hat{\bT}- \hat{\bV}_0^{\hat{q},\bU,{\hat z}_r}\|_{\infty}$ over all $\bU\in \{-1,0,1\}^{k\times k}$ since both $\hat{\bT}$ and $\hat{\bV}_0^{\hat{q},\bU,{\hat z}_r}$ are block-constant matrices with the same block structure determined by ${\hat z}_r$.

Denoting by $\bf 0$ the zero $k\times k$ matrix we have
\beqn
\|\hat{\bT}-\hat{\bV}\|_F&\leq& \|\hat{\bT}- \hat{\bV}_0^{\hat{q},{\bf 0},{\hat z}_r}\|_F 
\quad \quad \quad  \quad  \text{ (since $\hat U$ is a minimizer of $\|\hat{\bT}- \hat{\bV}_0^{\hat{q},\bU,{\hat z}_r}\|_F$)}
\\
&=& \|\hat{\bT}-\epsilon_{\hat{q}}\hat{\bV}_0\|_F
\quad \quad \quad  \quad  \quad \text{ (since $\|\hat{\bT}\|_{\infty}\leq 2r$)}
\\
&=& \|\hat{\bT}-\epsilon_{\hat{q}}\hat{\bT}/\|\hat{\bT}\|_F\|_F+\epsilon_{\hat{q}} \|\hat{\bT}/\|\hat{\bT}\|_F- \hat{\bV}_0\|_F 
\\ 
&\leq& (\|\hat{\bT}\|_F-\epsilon_{\hat{q}})+ \epsilon_{\hat{q}} \|\hat{\bT}/\|\hat{\bT}\|_F- \hat{\bV}_0\|_F 
\\
&\leq & (\|\hat{\bT}\|_F-\epsilon_{\hat{q}})+ \frac{\epsilon_{\hat{q}}}{4}\quad \quad \quad   \text{ (since $\cC_{\hat{z}_r}$ is a $1/4$-net)}\\
&\leq & \|\hat{\bT}\|_F/4.
\eeqn
Next, since $\hat{\bV} = \hat{\bV}_0^{\hat{q},\hat\bU,{\hat z}_r}$ minimizes $\|\hat{\bT}- \hat{\bV}_0^{\hat{q},\bU,{\hat z}_r}\|_{\infty}$ over $\bU$ we have 
$$
\|\hat{\bT}-\hat{\bV}\|_\infty \le \|\hat{\bT}- \hat{\bV}_0^{\hat{q},\bU^\ast,{\hat z}_r}\|_{\infty}
$$
where 
 $\bU^*$ is the matrix with elements defined by the relation $\bU^*_{ab}= {\rm sign}(\hat{\bT}_{ij})$ if $i\in{\hat z}_r^{-1}(a), j \in {\hat z}_r^{-1}(b)$ for $(a,b) \in [k]\times [k]$. Thus, all the entries of $\hat{\bV}_0^{\hat{q},\bU^\ast,{\hat z}_r}$ in each block are either equal to $r$ or to $-r$ depending on whether the value of $\hat{\bT}$ on this block is positive or negative respectively. Since $\|\hat{\bT}\|_{\infty}\leq 2r$ we obtain that $\|\hat{\bT}- \hat{\bV}_0^{\hat{q},\bU^\ast,{\hat z}_r}\|_{\infty} \le r$.

\end{proof}

{\bf Acknowledgement} 
The work of O. Klopp was conducted as part of the project Labex MME-DII (ANR11-LBX-0023-01).
The work of A.B. Tsybakov was supported by GENES and by the French National Research Agency (ANR) under the grants 
IPANEMA (ANR-13-BSH1-0004-02), and Labex ECODEC (ANR - 11-LABEX-0047). The work of N.~Verzelen was partly supported by the ANR under the grant ANR-2011-BS01-010-01, project Calibration. This research benefited from the support of the ``Chaire Economie et Gestion des Nouvelles Donn\'ees", under the auspices of Institut Louis Bachelier, Havas-Media and Paris-Dauphine.  The authors are grateful to Jiaming Xu for a comment on Corolary 2.7 
in a previous version of this  manuscript.

\bibliographystyle{plain}

\bibliography{ref}

%
%
%
%
%
%
%
%
%
%
%

\end{document}